\pdfoutput=1
\RequirePackage{ifpdf}
\ifpdf 
\documentclass[pdftex]{sigma}
\else
\documentclass{sigma}
\fi

\usepackage{mathrsfs}

\numberwithin{equation}{section}

\newtheorem{Theorem}{Theorem}[section]
\newtheorem*{Theorem*}{Theorem}

\newtheorem{Lemma}[Theorem]{Lemma}
\newtheorem{Claim}[Theorem]{Claim}
\newtheorem{Proposition}[Theorem]{Proposition}
 { \theoremstyle{definition}

\newtheorem{Remark}[Theorem]{Remark} }

\begin{document}

\allowdisplaybreaks

\renewcommand{\thefootnote}{}

\newcommand{\arXivNumber}{2301.12385}

\renewcommand{\PaperNumber}{070}

\FirstPageHeading

\ShortArticleName{Symplectic Double Extensions for Restricted Quasi-Frobenius Lie (Super)Algebras}

\ArticleName{Symplectic Double Extensions for Restricted\\ Quasi-Frobenius Lie (Super)Algebras\footnote{This paper is a~contribution to the Special Issue on Differential Geometry Inspired by Mathematical Physics in honor of Jean-Pierre Bourguignon for his 75th birthday. The~full collection is available at \href{https://www.emis.de/journals/SIGMA/Bourguignon.html}{https://www.emis.de/journals/SIGMA/Bourguignon.html}}}

\Author{Sofiane BOUARROUDJ~$^{\rm a}$, Quentin EHRET~$^{\rm b}$ and Yoshiaki MAEDA~$^{\rm c}$}

\AuthorNameForHeading{S.~Bouarroudj, Q.~Ehret and Y.~Maeda}

\Address{$^{\rm a)}$~Division of Science and Mathematics, New York University Abu Dhabi,\\
\hphantom{$^{\rm a)}$}~P.O.~Box 129188, Abu Dhabi, United Arab Emirates}
\EmailD{\href{mailto:sofiane.bouarroudj@nyu.edu}{sofiane.bouarroudj@nyu.edu}}

\Address{$^{\rm b)}$~University of Haute-Alsace, IRIMAS UR 7499, F-68100 Mulhouse, France}
\EmailD{\href{mailto:quentin.ehret@uha.fr}{quentin.ehret@uha.fr}}

\Address{$^{\rm c)}$~Tohoku Forum for Creativity, Tohoku University,
 2-1-1, Katahira, Aoba-ku, Sendai, Japan}
\EmailD{\href{mailto:ymkeiomath@gmail.com}{ymkeiomath@gmail.com}}

\ArticleDates{Received January 31, 2023, in final form September 11, 2023; Published online September 28, 2023}

\Abstract{In this paper, we present a method of symplectic double extensions for restricted quasi-Frobenius Lie superalgebras. Certain cocycles in the restricted cohomology represent obstructions to symplectic double extension, which we fully describe. We found a necessary condition for which a restricted quasi-Frobenius Lie superalgebras is a symplectic double extension of a smaller restricted Lie superalgebra. The constructions are illustrated with a~few examples.}


\Classification{17B50; 17B20}

\renewcommand{\thefootnote}{\arabic{footnote}}
\setcounter{footnote}{0}

\section{Introduction} 

A Lie group $G$ is called symplectic if it has a left-invariant closed 2-form whose rank is equal to $\dim(G)$. In this case, $\mathfrak{g} = \mathrm{Lie}(G)$ would be a quasi-Frobenius Lie algebra. Specifically, there exists $\omega \in Z^2(\mathfrak{g})$ such that $\omega$ is non-degenerate, see Section \ref{qFrobenius} for more details. Symplectic Lie groups include, for example, abelian groups of even dimension and the group of affine transformations of ${\mathbb K}^n$, for $n\geq 1$, where ${\mathbb K}={\mathbb R}$ or ${\mathbb C}$, see \cite{BC, MR2}.

Medina and Revoy classified symplectic nilpotent groups by their Lie algebras in \cite{MR2}. Symplectic double extensions were introduced in their work, and they showed that each nilpotent quasi-Frobenius Lie algebra can be obtained by a sequence of symplectic double extensions from the trivial Lie algebra. As a group, $(G, \omega)$ is a double extension of the group $(H, \Omega)$, meaning~$(H, \Omega)$ is a Marsden--Weinstein reduced manifold of $(G, \omega)$. Therefore, symplectic nilmanifolds can be obtained by symplectic double extensions.

The notion of symplectic double extension is the symplectic analog to the notion of double extension studied earlier in \cite{FS, MR1} when the bilinear form is non-degenerate, invariant and symmetric (abbreviated NIS in \cite{BKLS}). In contrast to usual double extensions, there are some cohomological obstructions to carry out symplectic double extensions, as discussed in \cite{MR2}.

The construction was later generalized to arbitrary symplectic Lie groups by Dardi\'e and Medina \cite{DM1}. They showed that such groups can be described in terms of semi-direct products of Lie groups or symplectic reduction and principal fiber bundles with affine fiber. Furthermore, they showed that each group obtained by this process carries an invariant Lagrangian foliation such that the affine structure defined by the symplectic form over each leaf is complete with respect to the affine structure deduced from that of the symplectic group.

In \cite{DM1}, several examples of symplectic double extensions were provided. In particular, symplectic Lie groups containing a closed normal subgroup of codimension 1 that is either the abelian group or the Heisenberg group are obtained through the process of symplectic double extension.\looseness=1

The structure theory of symplectic Lie groups was further studied in depth by Baues and Cort\'es in \cite{BC}. In their work, they introduced the concepts of symplectic {\it reduction} and symplectic {\it oxidation}. Basically, these two notions reverse each other, and the notion of symplectic oxidation is close to the notion of generalized double extensions (compare \cite[Section 2.3]{BC} and~\cite[Th\'eor\`eme~2.7]{DM1}). It was shown that every symplectic Lie group admits a sequence of subsequent symplectic reductions leading to a unique irreducible symplectic Lie group (referred to as the {\it symplectic basis}). Such a result was described in \cite{BC} as a Jordan--H\"older type uniqueness theorem. In the Lie algebra context, symplectic reduction is associated with the study of a reduction with respect to isotropic ideals. Further, it was shown that the isomorphism class of the simply connected symplectic basis is independent of the chosen reduction sequence. Another remarkable result of \cite{BC} is that all symplectic Lie groups of dimension less than or equal to 6 have a~Lagrangian subgroup; moreover, there is an eight-dimensional symplectic Lie group without a~Lagrangian subgroup.

In \cite{BFLM}, Bazzoni et al.\ introduced the notion of complex symplectic oxidation, and obtained certain complex symplectic Lie algebras of dimension $4n +4$ from those of dimension $4n$. The method has been used to describe all nilpotent complex symplectic Lie algebras in eight dimensions.

Recently, Bajo and Benayadi in \cite{BaBe}, generalized the idea of symplectic double extensions studied in \cite{DM2} even when the center is trivial in the case of an abelian para-K\"ahler Lie algebra~$\mathfrak{g}$ which is automatically solvable. They also showed that such a Lie algebra can be obtained by a~sequence of generalized double extensions from the affine Lie algebra. In fact, the case where~$\mathfrak{g}$ is solvable but not necessarily abelian para-K\"ahler has not yet been solved.

It is shown in \cite{BM} that these results can be superized for Lie superalgebras. In this case, there are four cases to consider: either the bilinear form or the derivation can be even or odd. It was also shown in \cite{BM} that Filiform Lie superalgebras are quasi-Frobenius, and the 4-dimensional quasi-Frobenius Lie superalgebras are classified based on Backhouse's classification \cite{Ba}. As far as we know, nothing has been done about supergroups.

Lie groups and Lie algebras play an important role in geometry, they provide relevant examples of manifolds and an algebraic characterization of some structures.

\subsection{Restricted Lie superalgebras and double extensions}
As far as we know, Jacobson introduced the concept of a restricted Lie algebra
\cite{J}. Roughly speaking, one requires the existence of an endomorphism on the modular
Lie algebra that resembles the $p$-th power mapping $x \mapsto x^p$ in associative algebras. Lie algebras associated with algebraic groups over fields with positive characteristic are restricted, and this class resembles the characteristic 0 case, see \cite{S, SF}. Superization
of the notion of restrictedness was studied by several authors, see, for example, \cite{BKLS, Fa, P,SZ, U,Y, YCC, Z} and especially \cite{BLLS}, where new phenomena were observed in characteristic 2. The restrictedness of simple (and close to simple) Lie superalgebras with Cartan matrices was fully studied in \cite{BKLLS}.

Cohomology theory for restricted Lie algebras was introduced by Hochschild \cite{H}. While higher order cohomology is difficult to calculate, it can be calculated at lower order, see \cite{EF}. An attempt of a superization of the restricted cohomology was given in \cite{YCC}. These constructions will be used in this paper to describe some restricted cocycles and their triviality.

The double extensions of Lie superalgebras equipped with a NIS (also called quadratic Lie superalgebra) was initiated in \cite{BB}, then followed by a series of papers refining the results, see~\cite{B} and references therein. Among the best known examples are the affine Kac--Moody algebras, which are double extensions of loop algebras. In the supersetting, the challenge is to prove that every Lie superalgebra with a NIS is a double extension of a smaller Lie superalgebra, since the Levi decomposition does not hold for Lie superalgebras. Moreover, the double extension of modular Lie superalgebras in characteristic 2 requires new features as demonstrated in \cite{BeBou1, BeBou2}. The double extension of {\it restricted} Lie superalgebras equipped with NIS was studied recently in~\cite{BBH}, and several examples, mainly Lie superalgebras with Cartan matrices, were studied based on \cite{BKLS, BKLLS}. This paper is the `symplectic' analog of the paper \cite{BBH}. Here, we construct double extensions for restricted quasi-Frobenius Lie superalgebras, and find necessary conditions for the double extension to be restricted too. The main novelty here is that the obstructions to carrying out symplectic double extensions are captured by the restricted cohomology, in contrast to the `non-restricted' case, where the obstructions are captured by the Chevalley--Eilenberg cohomology, see Theorems \ref{MainTh}, \ref{MainThO}, \ref{MainThOE}, and \ref{MainThOO}. With only 1-cocycles or 2-cocycles to deal with, the restricted cohomology will be quite manageable.

This is how the paper is structured. Basic concepts of restricted Lie superalgebras, cohomology of restricted Lie superalgebras, and quasi-Frobenius structures are discussed in Section~\ref{sec2}. To make the paper self-contained, we have recalled some concepts, even though they are well-known. Sections \ref{results} and \ref{sec4} is devoted to the construction of symplectic double extensions. The main results are given in Theorems \ref{MainTh}, \ref{MainThO}, \ref{MainThOE}, and \ref{MainThOO}. The converse of these theorems are given by Theorems \ref{Rec1}, \ref{Rec2}, \ref{Rec3} and \ref{Rec4}. In Section \ref{sec5}, we provide a few examples of symplectic double extensions, borrowed from \cite{BM, GKN}.

\section{Restricted Lie superalgebras and restricted cohomology}\label{sec2}
Hereafter, $\mathbb{K}$ is an arbitrary field of characteristic $\mathrm{char}(\mathbb{K})=p>2$. We refer to \cite{SF} for a~thorough study of restricted Lie algebras.

\subsection{Restricted Lie algebras}
Let $\mathfrak{a}$ be a~finite-dimensional modular Lie algebra over~$\mathbb{K}$. Following \cite{J, SF}, a~map $[p]\colon \mathfrak{a}\rightarrow \mathfrak{a}$, $a\mapsto a^{[p]}$ is called a~\textit{$p$-structure} of $\mathfrak{a}$ and $\mathfrak{a}$ is said to be {\it restricted} if
\begin{gather*}
\mathrm{ad}_{a^{[p]}}=(\mathrm{ad}_a)^p \qquad \text{for all} \ a \in \mathfrak{a},\\
(\alpha a)^{[p]}=\alpha^p a^{[p]}\qquad \text{for all}\ a\in \mathfrak{a} \ \text{and} \ \alpha \in \mathbb{K},\\
(a+b)^{[p]}=a^{[p]}+b^{[p]}+\sum_{1\leq i\leq p-1}s_i(a,b),
\end{gather*}
where the $s_i(a,b)$ can be obtained from
\begin{gather*}
(\mathrm{ad}_{\lambda a+b})^{p-1}(a)= \sum_{1\leq i \leq p-1} is_i(a,b) \lambda^{i-1}.
\end{gather*}
To investigate $p$-structures on a Lie algebra, it is useful to use the following theorem, due to Jacobson.
\begin{Theorem}[\cite{J}] \label{Jac}
Let $(e_j)_{j\in J}$ be a~basis of $\mathfrak{a}$ such that there are $f_j\!\!\in \!\mathfrak{a}$ satisfying ${(\mathrm{ad}_{e_j})^p\!=\!\mathrm{ad}_{f_j}}$. Then, there exists exactly one $[p]$-map $[p]\colon\mathfrak{a}\rightarrow \mathfrak{a}$ such that
\[
e_j^{[p]}=f_j \qquad \text{for all}\ j\in J.
\]
\end{Theorem}
In the sequel, we will need the following lemma.
\begin{Lemma}\label{1lemma}
The following equality holds
\begin{align*}\label{si}
\sum_{i=1}^{p-1}s_i(a,b)&=\sum_{\underset{x_{p-1}=b,~x_p=a}{x_k=a \text{ or } b}}\frac{1}{\pi(a)}[x_1,[x_2,[\dots ,[x_{p-1},x_p]\dots ]]],
\end{align*}
where $\pi\{a\}$ stands for the number of terms $x_k$ equal to $a$.
\end{Lemma}
\begin{proof}
 One can check, by expanding the expression $(\mathrm{ad}_{\lambda a+b})^{p-1}(a)$ and noticing that $[\lambda a+b,a]=[b,a]$, that the coefficient of $\lambda^{i-1}$ is given by
 \begin{equation}
 \label{sibracket}
 \sum_{\underset{\sharp\{k,~x_k=a\}=i-1}{x_k=a \text{ or } b}}[x_1,[x_2,[\dots ,[x_{p-2},[b,a]]\dots ]]].
 \end{equation}
The expression (\ref{sibracket}) contains $i$ variables ``$a$'', so $i=\mathrm{card}\{k,\,x_k=a\}+1. $
It follows that
\begin{align*}
 s_i(a,b)&=\sum_{\underset{\text{card}\{k,~x_k=a\}=i-1}{x_k=a \text{ or } b}}\frac{1}{i}[x_1,[x_2,[\dots ,[x_{p-2},[b,a]]\dots ]]]\\[-1mm]
 &=\sum_{\underset{\text{card}\{k,~x_k=a\}=i-1}{x_k=a \text{ or } b}}\frac{1}{\pi\{a\}}[x_1,[x_2,[\dots ,[x_{p-2},[b,a]]\dots ]]].
\end{align*}
Therefore,
\begin{align*}
 \sum_{i=1}^{p-1}s_i(a,b)&=\sum_{i=1}^{p-1}\sum_{\underset{\text{card}\{k,~x_k=a\}=i}{x_k=a \text{ or } b}}\frac{1}{\pi\{a\}}[x_1,[x_2,[\dots ,[x_{p-2},[b,a]]\dots ]]]\\[-1mm]
 &=\sum_{\underset{x_{p-1}=b,~x_p=a}{x_k=a \text{ or } b}}\frac{1}{\pi\{a\}}[x_1,[x_2,[\dots ,[x_{p-1},x_p]\dots ]]]. \tag*{\qed}
\end{align*}
\renewcommand{\qed}{}
\end{proof}
Let $\mathfrak{a}$ be a restricted Lie algebra. An $\mathfrak{a}$-module $M$ is called \textit{restricted} if
\[
\underbrace{a\cdots a}_{p\text{~~times}}\cdot m =a^{[p]}\cdot m \qquad \text{for all}\ a\in \mathfrak{a} \ \text{and any} \ m\in M.
\]
\subsection{Restricted Lie superalgebras}
Let $\mathfrak{a}=\mathfrak{a}_{\bar 0}\oplus \mathfrak{a}_{\bar 1} $ be a~finite-dimensional Lie superalgebra defined over a field of characteristic~$p>2$. Since $[a,[a,a]]=0$, for all $a\in \mathfrak{a}_{\bar 1}$, does not follow from the Jacobi identity for $p=3$, we require it as part of our definition. In \cite{BBH}, we presented two examples of superalgebras for which~$[a,[a,a]]\not=0$, where $a\in \mathfrak{a}_{\bar 1}$, constructed by means of double extensions, and we called them pre-Lie superalgebras. We denote the parity of a given non-zero homogeneous element~$a\in\mathfrak{a}$ by~$|a|$.\looseness=-1

Following \cite{P}, we say that $\mathfrak{a}$ has a~\textit{$p|2p$-structure} if $\mathfrak{a}_{\bar 0}$ is a restricted Lie algebra and
\[
\mathrm{ad}_{a^{[p]}}(b)=(\mathrm{ad}_a)^p(b)\qquad \text{for all}\ a \in \mathfrak{a}_{\bar 0}\ \text{and}\ b\in \mathfrak{a}.
\]

Recall that the bracket of two odd elements of the Lie superalgebra is polarization of squaring~$a\mapsto a^2$ \big(namely, $[a,b]=(a+b)^2-a^2-b^2$\big). We set{\samepage
\[
[2p]\colon\ \mathfrak{a}_{\bar 1} \rightarrow \mathfrak{a}_{\bar 0}, \qquad a\mapsto \bigl(a^2\bigr)^{[p]}\qquad \text{for any} \ a\in \mathfrak{a}_{\bar 1}.
\]
The pair $(\mathfrak{a}, [p|2p])$ is referred to as a~\textit{restricted} Lie superalgebra.}

For $p=2$, there are several notions of restrictedness, see \cite{BLLS}. We do not consider them in this paper.

The following theorem is a~straightforward superization of Jacobson's Theorem~\ref{Jac}.
\begin{Theorem}
Let $(e_j)_{j\in J}$ be a~basis of $\mathfrak{a}_{\bar 0}$, let the elements $f_j\in \mathfrak{a}_{\bar 0}$ be such that ${(\mathrm{ad}_{e_j})^p=\mathrm{ad}_{f_j}}$. Then, there exists exactly one $[p|2p]$-mapping $[p|2p]\colon\mathfrak{a}\rightarrow \mathfrak{a}$ such that
\[
e_j^{[p]}=f_j \qquad \text{for all} \ j\in J.
\]
\end{Theorem}
Additionally, Lemma \ref{1lemma} is also valid for restricted Lie superalgebras.

A homogeneous ideal $I=I_{\bar 0}\oplus I_{\bar 1}$ of $\mathfrak{a}$ is called a~\textit{$p$-ideal} if it is closed under the $[p|2p]$-map; namely
\[
a^{[p]}\in I_{\bar 0}\qquad \text{for all}\ a\in I_{\bar 0}.
\]
As $a^{[2p]}=\big(a^2\big)^{[p]}$, for all $a\in \mathfrak{a}_{\bar 1}$, this would imply that
\[
a^{[2p]}\in I_{\bar 0}\qquad \text{for all}\ a\in \mathfrak{a}_{\bar 1}.
\]
Let $\mathfrak{a}$ be a restricted Lie superalgebra. An $\mathfrak{a}$-module $M$ is called \textit{restricted} if
\[
\underbrace{a\cdots a}_{p\text{~~times}} \cdot m =a^{[p]}\cdot m \qquad \text{for all}\ a\in \mathfrak{a}_{\bar 0} \ \text{and}\ m\in M.
\]
Consequently,
\[
\underbrace{a\cdots a}_{2p\text{~~times}} \cdot m =a^{[2p]}\cdot m \qquad \text{for all}\ a\in \mathfrak{a}_{\bar 1} \ \text{and}\ m\in M.
\]
\subsection{Restricted cohomology for restricted Lie superalgebras}
Hochschild introduced restricted cohomology \cite{H}. By defining the restricted cohomology groups of restricted Lie algebras with the restricted universal enveloping algebra introduced by Jacobson \cite{J}, Hochschild was able to establish a connection to the ordinary Chevalley--Eilenberg cohomology of Lie algebras with a six-term exact sequence \cite{H}. Contributions were also made by May \cite{M}. Hochschild's definition is not practical for computation, however.
Fuchs and Evans addressed this issue by providing a new free resolution of the ground field in terms of modules over the restricted enveloping algebra. Cohomology groups can be built up to order $p$ if the restricted Lie algebra is abelian using this approach, which only allows them to describe the first and second restricted cohomology groups in the non-abelian case due to the Lie bracket and non-linearity of the $[p]$-map \cite{EF}. This cohomology was computed explicitly and further examined by Evans and Fialowski, see \cite{EF4, EF5} and references therein. The superization of the restricted cohomology of Lie algebras was given in \cite{YCC}. In this section, we review the basic concepts that are only required for our construction.\looseness=1

Let $\mathfrak{a}$ be a restricted Lie superalgebra, and let $M$ be a restricted $\mathfrak{a}$-module. Let us review the Chevalley--Eilenberg cohomology for Lie superalgebras, see \cite{BGL} for more details.

Let us define the space of cochains $C^n_{\text{CE}}(\mathfrak{a}; M)$ in the Chevalley--Eilenberg cohomology. For ${n=0}$, we put $C^0_{\text{CE}}(\mathfrak{a}; M):=M$. For $n>0$, the space of cochains $C^n_{\text{CE}}(\mathfrak{a}; M)$ is the space of $n$-linear super anti-symmetric maps. The differential is given as follows:
\begin{gather*}
 d_{\mathrm{CE}}^0(m)(x)=(-1)^{|m||x|}x\cdot m\qquad\text{for any}\ m\in M\ \text{and} \ x\in\mathfrak{a},\\
d_{\mathrm{CE}}^n(\varphi)(x_1,\dots, x_n) = \mathop{\sum}\limits_{i<j}(-1)^{(|x_j|(|x_{i+1}|+\dots+|x_{j-1}|)+j}\\
\phantom{d_{\mathrm{CE}}^n(\varphi)(x_1,\dots, x_n) =}{}\times\varphi(x_1, \dots, x_{i-1}, [x_i,x_j], x_{i+1} \dots, \widetilde x_j, \dots, x_n)\\
\phantom{d_{\mathrm{CE}}^n(\varphi)(x_1,\dots, x_n) =\times}{}+\mathop{\sum}\limits_{j}(-1)^{(|x_j|)(|\varphi|+|x_{1}|+\dots+|x_{j-1}|)+j}x_j  \varphi (x_1, \dots, \widetilde x_j, \dots, x_n)
\end{gather*}
for any $\varphi \in C^{n-1}_{\text{CE}}(\mathfrak{a}; M)$ with $n>0$, and $x_1,\dots, x_n\in\mathfrak{a}$.

Let $\varphi\in C^2_{\text{CE}}(\mathfrak{a};M)$ and $\theta\colon \mathfrak{a}_0\rightarrow M$ be two maps.
	We say that $\theta$ \textit{has the $*$-property with respect to} $\varphi$ if
\begin{gather*}
		\theta(\delta a)=\delta^p\theta(a),\qquad \text{for all}\ \delta\in \mathbb K \ \text{and}\ a\in \mathfrak{a}_{\bar 0},\\
		\theta(a+b)=\theta(a)+\theta(b)\\
\phantom{\theta(a+b)=}{}+ \sum_{\underset{x_1=a,~x_2=b}{x_i=a \text{ or } b}} \frac{1}{\pi\{a\}}\sum_{k=0}^{p-2}(-1)^kx_p\dots x_{p-k+1}\varphi([[\dots [x_1,x_2],x_3]\dots ,x_{p-k-1}],x_{p-k}),
\end{gather*}
with $a,b\in \mathfrak{a}_{\bar 0}$, and $\pi\{a\}$ stands for the number of factors $x_i$ that are equal to $a$.

We are ready now to define the space of restricted cochains as
\begin{gather*}
C^1_*(\mathfrak{a};M) := C^1_{\text{CE}}(\mathfrak{a};M),\\
C^2_*(\mathfrak{a};M) := \big\lbrace (\varphi,\theta),\, \varphi\in C^2_{\text{CE}}(\mathfrak{a};M),\,\theta\colon \mathfrak{a}_{{\bar 0}}\rightarrow M \text{ has the $*$-property w.r.t.} \ \varphi \big\rbrace.
\end{gather*}

An element $\varphi \in C^1(\mathfrak{a}; M)$ induces a map
\begin{align*}
\mathrm{Ind}^1(\varphi)\colon\ \mathfrak{a} &\longrightarrow M,\\
x&\longmapsto \varphi\bigl(x^{[p]}\bigr)-x^{p-1}\varphi(x).
\end{align*}

An element $(\alpha;\beta)\in C_*^2(\mathfrak{a},M)$ induces a map
\begin{align*}
	\mathrm{Ind}^2(\alpha,\beta)\colon\ \mathfrak{a}_0 \times \mathfrak{a}_0 & \longrightarrow M, \\
	(a,b)&\longmapsto \alpha\bigl( a,b^{[p]}\bigr) - \sum_{i+j=p-1}(-1)^ib^i \alpha\bigl( [[\dots [a,\overset{j\text{ terms}}{\overbrace{b],\dots ],b}]},b\bigr)+a\beta(b).
	\end{align*}

A restricted $2$-cocycle is an element $(\alpha;\beta)\in C_*^2(\mathfrak{a};M)$ such that
\[
\bigl(d_{\rm CE}^2\alpha,\mathrm{Ind}^2(\alpha,\beta)\bigr)=(0,0).
\]
We denote by $Z_*^2(\mathfrak{a};M)$ the space of restricted 2-cocycles.

An even restricted $2$-coboundary is an element $(\alpha; \beta)\in C_*^2(\mathfrak{a};M)$ such that there exists $\varphi\in C_*^1(\mathfrak{a}; M)=C_{\text{CE}}^1(\mathfrak{a}; M)$ for which
\begin{gather*}
\alpha(a,b) = \varphi([a,b])-a\varphi(b)+b\varphi(a) \qquad \text{for all}\ a,b\in \mathfrak{a}; \\
\beta(a) = \varphi\bigl(a^{[p]}\bigr)-a^{p-1}\varphi(a) \qquad \text{for all} \ a\in \mathfrak{a}_{\bar 0}.
\end{gather*}
We denote by $B_*^2(\mathfrak{a};M)$ the space of restricted 2-coboundaries.

\subsection{Quasi-Frobenius Lie superalgebras}\label{qFrobenius}

As far as we know, the notion of quasi-Frobenius Lie algebra is due to Seligman. Those Lie algebras have been introduced to answer a question raised by Jacobson: If $\mathfrak{a}$ is finite-dimensional,
what conditions one can put on $\mathfrak{a}$ in order the enveloping algebra $U(\mathfrak{a})$ admits an
exact simple module? The answer to this question was settled by Ooms in \cite{O}.

Several authors refer to quasi-Frobenius Lie algebras as symplectic Lie algebras, see \cite{AM, BC,BFLM,DM2,DM1, D, Fi, GR, MR2}.

The superization of this notion is immediate. We have preferred to retain the term ‘quasi-Frobenius’ in the super setting.

Following \cite{BaBe, BM}, a Lie superalgebra $\mathfrak{a}$ is called {\it quasi-Frobenius} if it is equipped with a 2-cocycle $\omega \in Z^2_ {\text{CE}}(\mathfrak{a}; \mathbb{K})$ such that $\omega$ is a non-degenerate bilinear form. Explicitly, for all $a,b\in \mathfrak{a}$ we have
\begin{gather*}
(-1)^{|a||c|} \omega(a, [b , c]) +(-1)^{|c||b|} \omega(c, [a , b] )+(-1)^{|b||a|} \omega(b, [c , a] )=0.
\end{gather*}

We denote such an algebra by $(\mathfrak{a}, \omega)$. For a list of quasi-Frobenius Lie superalgebras, see \cite{BM}.

In the case where $\omega \in B^2(\mathfrak{a},\mathbb K)$, the Lie superalgebra $\mathfrak{a}$ is called a {\it Frobenius} Lie superalgebra.

Recall that a bilinear form $\omega$ is called {\it even} (resp, {\it odd}) if $\omega(\mathfrak{a}_{\bar 0}, \mathfrak{a}_{\bar 1})= \omega(\mathfrak{a}_{\bar 1},\mathfrak{a}_{\bar 0})=0$ (resp.\ if $\omega(\mathfrak{a}_{\bar 0}, \mathfrak{a}_{\bar 0})= \omega(\mathfrak{a}_{\bar 1},\mathfrak{a}_{\bar 1})=0$).

Recall that an even (resp.\ odd) non-degenerate bilinear form is called {\it orthosymplectic} (resp.\ {\it periplectic}). This leads us to the following definition:

A quasi-Frobenius Lie superalgebra $(\mathfrak{a}, \omega )$ is called {\it orthosymplectic quasi-Frobenius} (resp.\ {\it periplectic quasi-Frobenius}) if the form $\omega$ is even (resp.\ odd) on $\mathfrak{a}$.

Let $(\mathfrak{a}, \omega)$ be a quasi-Frobenius Lie superalgebra and let $S \subseteq \mathfrak{a}$ be a subspace. The orthogonal of $S$ in $(\mathfrak{a}, \omega)$ is
\[
S^\perp = \{v \in \mathfrak{a}\mid \omega(v,w) = 0\ \text{for all} \ w \in S\}.
\]
The subspace $S$ is called {\it non-degenerate} if $S \cap S^\perp = \{0\}$. It is called
{\it isotropic} if ${S \subset S^\perp}$. A~maximal isotropic subspace is called {\it Lagrangian} if it satisfies $S = S^\perp$.

If $I$ is an ideal of $\mathfrak{a}$, then $I^\perp$ is an ideal if and only if $I^\perp \subseteq Z_\mathfrak{a}(I)$, where $Z_\mathfrak{a}(I)$ stands for
the centralizer of $I$ in $\mathfrak{a}$ (see \cite[Lemma~2.4]{BM}).

\subsection{The adjoint of a derivation}
Consider a linear map ${\mathscr D}\in \mathrm{End}(\mathfrak{a})=\mathrm{End}(\mathfrak{a})_{\bar 0}\oplus \mathrm{End}(\mathfrak{a})_{\bar 1} $. The map ${\mathscr D}$ is called a derivation of~$\mathfrak{a}$ if it satisfies the following condition:
\[
{\mathscr D}([a,b])=[ {\mathscr D}(a),b] +(-1)^{|{\mathscr D}||a|} [a, {\mathscr D} (b)]\qquad \text{for all}\ a, b\in \mathfrak{a}_i,\ \text{where}\ i=\bar 0, \bar 1.
\]
Let us denote by $\mathfrak{der}(\mathfrak{a})$ the space of all derivations on $\mathfrak{a}$.

To each derivation ${\mathscr D}\in \mathfrak{der}(\mathfrak{a})$ one can assign a unique linear map ${\mathscr D}^*$ with respect to the non-degenerate form $\omega$, called the {\it adjoint} of ${\mathscr D}$, that satisfies the following condition:
\[
\omega({\mathscr D}(a),b) = (-1)^{|a| |{\mathscr D}|} \omega(a, {\mathscr D}^*(b)).
\]
By construction of ${\mathscr D}^*$, it is easy to see that ${\mathscr D}^*$ is a derivation as well.

Let us assume now that $\mathfrak{a}$ is restricted. A derivation $\mathscr{D}\in \mathfrak{der}(\mathfrak{a})$ is called \textit{restricted} if
\[
\mathscr{D}\bigl(a^{[p]}\bigr)=(\mathrm{ad}_a)^{p-1} (\mathscr{D}(a)) \qquad \text{for all} \ a\in \mathfrak{a}_{\bar 0}.
\]
Consequently, we have
\[
\mathscr{D}\bigl(a^{[2p]}\bigr)=(\mathrm{ad}_{a^2})^{p-1} \bigl(\mathscr{D}\big(a^2\big)\bigr) \qquad \text{for all} \ a\in \mathfrak{a}_{\bar 1}.
\]
The space of restricted derivations is denoted by $\mathfrak{der}^p(\mathfrak{a})$. Outer restricted derivations of nilpotent restricted Lie algebras has been investigated in \cite{FSW}.

Following \cite{BBH}, a derivation $\mathscr{D}\in \mathfrak{der}^p_{\bar 0}(\mathfrak{a})$ has \textit{$p$-property} if there exist $\gamma \in \mathbb{K}$ and $a_0\in \mathfrak{a}_{\bar 0}$ such that
\begin{equation}
\label{SConp}
\mathscr{D}^p=\gamma \mathscr{D}+\mathrm{ad}_{a_0},\qquad
\mathscr{D}(a_0)=0.
\end{equation}
For a list of outer derivations satisfying the $p$-property, see \cite{BBH}.
\subsection{A few useful cocycles} Let $(\mathfrak{a}, \omega)$ be a restricted quasi-Frobenius Lie superalgebra, and let ${\mathscr D}$ be a derivation on $\mathfrak{a}$.
\begin{Lemma}\label{Ccocycle}
Let us define the map
\begin{gather} \label{CEcocycle}
C\colon\ \mathfrak{a}\wedge \mathfrak{a} \rightarrow \mathbb K,\qquad (a,b)\mapsto \omega\bigl(\bigl({\mathscr D}+{\mathscr D}^*\bigr)(a),b\bigr)=\omega({\mathscr D}(a), b) +(-1)^{|a||\mathscr{D}|}\omega(a,{\mathscr D}(b)).
\end{gather}
Then $C\in Z^2_{\mathrm{CE}}(\mathfrak{a}; \mathbb{K})$. Moreover, if $\mathscr D$ is inner, then $C\in B^2_{\mathrm{CE}}(\mathfrak{a}; \mathbb{K})$.
\end{Lemma}

\begin{proof}
Using the fact that $\omega$ is a $2$-cocycle, we get
\begin{align*}
 \underset{a,b,c}{\circlearrowleft}(-1)^{|a||c|}C(a,[b,c])={}& \underset{a,b,c}{\circlearrowleft}(-1)^{|a||c|}\bigl(\omega({\mathscr D}(a),[b,c])+(-1)^{|{\mathscr D}||a|}\omega(a,[{\mathscr D}(b),c])\\
 &+(-1)^{|{\mathscr D}||a|+|{\mathscr D}||b|}\omega(a,[b,{\mathscr D}(c)])\bigr)\\
={}& \underset{{\mathscr D}(a),b,c}{\circlearrowleft}(-1)^{|{\mathscr D}(a)||c|}\omega({\mathscr D}(a),[b,c])+\underset{a,{\mathscr D}(b),c}{\circlearrowleft}\!(-1)^{|a||c|}\omega(a,[{\mathscr D}(b),c])\\
 &+\underset{a,b,{\mathscr D}(c)}{\circlearrowleft}(-1)^{p({\mathscr D}(c))|a|}\omega(a,[b,{\mathscr D}(c)])=0.
\end{align*}
Suppose now that $\mathscr D$ is inner; namely, $\mathscr D=\mathrm{ad}_X$, for some $X\in \mathfrak{a}$. Using the fact that $\omega$ is closed, a direct computation shows that
\[
C(a,b)=\omega(X, [a,b]). \tag*{\qed}
\]
 \renewcommand{\qed}{}
\end{proof}
The following definition is essential to us. Denote by $\sigma_i^\mathfrak{a}(a,b)$ the expression that appears in the following equation:
\begin{equation}\label{defsigma}
\omega\bigl(({\mathscr D}+{\mathscr D}^*)(\mu a+b),\bigl(\mathrm{ad}^\mathfrak{a}_{\mu a+b}\bigr)^{p-2}(a) \bigr)= \sum_{1\leq i \leq p-1}i \sigma_i^\mathfrak{a}(a,b)\mu^{i-1}.
\end{equation}
\begin{itemize}\itemsep=0pt
\item If $p=2$, then $\sigma_1^\mathfrak{a}(a,b)=\omega(({\mathscr D}+{\mathscr D}^*)(b),a).$
\item If $p=3$, then $\sigma_1^{\mathfrak{a}}(a,b)=\omega(({\mathscr D}+{\mathscr D}^*)(b),[b,a])$ and $\sigma_2^{\mathfrak{a}}(a,b)=2\omega(({\mathscr D}+{\mathscr D}^*)(a),[b,a])$.
\item If $p$ is any prime, then $\sigma_1^{\mathfrak{a}}(a,b)=\omega\bigl(({\mathscr D}+{\mathscr D}^*)(b),(\mathrm{ad}^{\mathfrak{a}}_b)^{p-2}(a)\bigr)$.
\end{itemize}
We will need the following lemmas.

\begin{Lemma}
\qquad
\begin{itemize}\itemsep=0pt
\item[$(i)$] For all $a,b\in \mathfrak{a}_{\bar 0}$, we have
\begin{gather*}
s_i^\mathfrak{g}(a,b)=
 \begin{cases}
s_i^\mathfrak{a}(a,b)+ \sigma_i^{\mathfrak{a}}(a,b)x & \text{if} \ |{\mathscr D}|+|\omega|=
{\bar 0}, \\
s_i^\mathfrak{a}(a,b) & \text{if} \ |{\mathscr D}|+|\omega|={\bar 1}.
\end{cases}
\end{gather*}
\item[$(ii)$] For every $a,b\in \mathfrak{a}_{\bar 0}$, we have
\begin{gather*}
\sigma_1^\mathfrak{a}({\mathscr D}(a),a) = \omega \bigl(({\mathscr D}+ {\mathscr D}^*)(a), \bigl(\mathrm{ad}_a^\mathfrak{a}\bigr)^{p-2}\circ {\mathscr D}(a)\bigr),\\
 \sigma_1^\mathfrak{a}([a,b], a) = \omega \bigl(({\mathscr D}+ {\mathscr D}^*)(a), \bigl(\mathrm{ad}_a^\mathfrak{a}\bigr)^{p-1}(b)\bigr).
\end{gather*}
\end{itemize}
\end{Lemma}
\begin{proof}
For part (i), since
\[
\big(\mathrm{ad}_{\mu a+b}^\mathfrak{g}\big)^{p-1}(a)=\big(\mathrm{ad}_{\mu a+b}^\mathfrak{a}\big)^{p-1}(a)+\omega\bigl(({\mathscr D}+{\mathscr D}^*)(\mu a+b),\big(\mathrm{ad}_{\mu a+b}^\mathfrak{a}\big)^{p-2}(a) \bigr)x,
\]
the result follows immediately. Part (ii) follows from the very definition of $\sigma^\mathfrak{a}$ given in equation~(\ref{defsigma}), by taking $\mu=0$ and choose $a$ and $b$ appropriately.
\end{proof}

\begin{Lemma} \label{sigmabra} For all $a,b\in \mathfrak{a}_{\bar 0}$, we have
\begin{gather*}
 \sum_{i=1}^{p-1}\sigma_i^{\mathfrak{a}}(a,b)=\sum_{\underset{x_{p-1}=b,~x_p=a}{x_k=a \text{ or } b}}\frac{1}{\pi\{a\}}\omega( ({\mathscr D}+{\mathscr D}^*)(x_1),[x_2,[\dots ,[x_{p-1},x_p]_{\mathfrak{a}}\dots ]_{\mathfrak{a}}]_{\mathfrak{a}} ).
\end{gather*}
\end{Lemma}
\begin{proof}
Similar to that of Lemma \ref{1lemma}.
\end{proof}
We are now ready to introduce another map $P$ satisfying certain conditions. The goal here is to construct a 2-cocycle $(C,P)\in Z^2_*(\mathfrak{a}; \mathbb K)$, as we have already shown in Lemma \ref{Ccocycle} that $C\in Z^2_{\text{CE}}(\mathfrak{a}; \mathbb K)$.This construction will be used on Section \ref{results} when we study symplectic double extensions. Let us define a map \begin{equation}
 \label{mapP} P\colon\ \mathfrak{a}_{\bar 0} \rightarrow \mathbb K, \qquad a\mapsto P(a)
 \end{equation}
satisfying the two conditions
\begin{gather}
P(\delta a) = \delta^p P(a)\qquad \text{for all}\ a \in \mathfrak{a}_{\bar 0}\ \text{and}\ \delta\in \mathbb F,\label{CondP1}\\
P(a+b) = P(a)+P(b)+ \sum_{i=1}^{p-1}\sigma_i^{\mathfrak{a}}(a,b) \qquad \text{for all} \ a, b \in \mathfrak{a}_{\bar 0}.\label{CondP2}
 \end{gather}

 \begin{Lemma} \label{*property}
The map $P$ has the $*$-property with respect to the $2$-cochain $C$ defined as in equation~\eqref{CEcocycle}.
\end{Lemma}

\begin{proof}
Using Lemma \ref{sigmabra}, it is enough to show that (for all $a,b\in \mathfrak{a}_{\bar 0}$)
\begin{align*}
 &\sum_{\underset{x_{p-1}=b,~x_p=a}{x_k=a \text{ or } b}}\frac{1}{\pi\{a\}}\omega( ({\mathscr D}+{\mathscr D}^*)(x_1),[x_2,[\dots ,[x_{p-1},x_p]_{\mathfrak{a}}\dots ]_{\mathfrak{a}}]_{\mathfrak{a}} )\\
 &= \sum_{\underset{x_1=a,~x_2=b}{x_i=a \text{ or } b}}\frac{1}{\pi\{a\}}\sum_{k=0}^{p-2}(-1)^kx_p\dots x_{p-k+1}\omega(({\mathscr D}+{\mathscr D}^*)([[\dots [x_1,x_2],x_3]\dots ,x_{p-k-1}]_{\mathfrak{a}}),x_{p-k}).
\end{align*}
Equivalently,
\begin{align*}
 \sum_{i=1}^{p-1}\sigma_i^{\mathfrak{a}}(a,b)
 =&\sum_{x_k=a \text{ or } b}\frac{1}{\pi\{a\}} \omega( ({\mathscr D}+{\mathscr D}^*)(x_1),[x_2,[x_3,[\dots ,[b,a]_{\mathfrak{a}}\dots ]_{\mathfrak{a}} ]_{\mathfrak{a}} )\\
 =&\sum_{x_k=a \text{ or } b}\frac{1}{\pi\{a\}} \omega( ({\mathscr D}+{\mathscr D}^*)(x_1),[x_2,[x_3,[\dots ,[b,a]_{\mathfrak{a}}\dots ]_{\mathfrak{a}} ]_{\mathfrak{a}} ]_{\mathfrak{a}} )\\
 =&\sum_{x_k=a \text{ or } b}\frac{1}{\pi\{a\}} \omega( ({\mathscr D}+{\mathscr D}^*)(x_p),[x_{p-1},[x_{p-2},[\dots ,[b,a]_{\mathfrak{a}}\dots ]_{\mathfrak{a}} ]_{\mathfrak{a}} ]_{\mathfrak{a}} )\\
 =&(-1)^{p-2}\sum_{x_k=a \text{ or } b}\frac{1}{\pi\{a\}} \omega( ({\mathscr D}+{\mathscr D}^*)(x_p),[[\dots [a,b]_{\mathfrak{a}},x_3]_{\mathfrak{a}},\dots ,x_{p-1}]_{\mathfrak{a}} )\\
 =&\sum_{x_k=a \text{ or } b}\frac{1}{\pi\{a\}} \omega( [[\dots [a,b]_{\mathfrak{a}},x_3]_{\mathfrak{a}},\dots ,x_{p-1}]_{\mathfrak{a}},({\mathscr D}+{\mathscr D}^*)(x_p) )\\
 =&\sum_{x_k=a \text{ or } b}\frac{1}{\pi\{a\}} \omega(({\mathscr D}+{\mathscr D}^*)( [[\dots [a,b]_{\mathfrak{a}},x_3]_{\mathfrak{a}},\dots ,x_{p-1}]_{\mathfrak{a}}),x_p )\\
 =&\sum_{\underset{x_1=a,~x_2=b}{x_i=a \text{ or } b}}\frac{1}{\pi\{a\}}\times\\
 &\quad\sum_{k=1}^{p-2}(-1)^kx_p\dots x_{p-k+1}\omega(({\mathscr D}+{\mathscr D}^*)([[\dots [x_1,x_2]_\mathfrak{a},x_3]_\mathfrak{a}\dots ,x_{p-k-1}]_{\mathfrak{a}}),x_{p-k})\\
 &\qquad+ \sum_{\underset{x_1=a,~x_2=b}{x_i=a \text{ or } b}}\frac{1}{\pi\{a\}}\omega(({\mathscr D}+{\mathscr D}^*)([[\dots [x_1,x_2]_\mathfrak{a},x_3]_\mathfrak{a}\dots ,x_{p-1}]_{\mathfrak{a}}),x_{p}).
\end{align*}
The last equality holds as every term in the inner sum \smash{$ \sum_{k=1}^{p-2}$} vanishes, since $\mathbb{K}$ is a trivial $\mathfrak{a}$-module.
\end{proof}

\begin{Proposition}\label{CP2cocycle}
Let $C$ be defined as in \eqref{CEcocycle}, and $P$ defines as in \eqref{mapP}. The $2$-cochain $(C,P)\in C^2_*(\mathfrak{a};\mathbb{K})$ is a restricted $2$-cocycle if and only if
\[
\omega(({\mathscr D}+{\mathscr D}^*)\bigl(b^{[p]}, a\bigr) = \omega\bigl(\bigl({\mathscr D}+{\mathscr D}^*\bigr)(b), \mathrm{ad}_b^{p-1}(a)\bigr)\qquad \text{ for all}\ a,b\in \mathfrak{a}_{\bar 0}.
\]
\end{Proposition}

\begin{proof}
We have already proved in Lemma \ref{Ccocycle} that $C\in Z^2_{\mathrm{CE}}(\mathfrak{a}; \mathbb{K})$, and that $P$ satisfies the $*$-property with respect to $C$ in Lemma \ref{*property}. It remains to check that $\mathrm{Ind}^2(C,P)(a,b)=0$, for every $a,b\in\mathfrak{a}_0.$ Indeed,
\begin{align*}
 \mathrm{Ind}^2(C,P)(a,b)&=C\bigl( a,b^{[p]}\bigr) - \sum_{i+j=p-1}(-1)^ib^i C\bigl( [[\dots [a,\overset{j\text{ terms}}{\overbrace{b],\dots ],b}]},b\bigr)+aP(b)\\
 &=C\bigl(a,b^{[p]}\bigr)-C\bigl([[\dots [a,\overset{p-1\text{ terms}}{\overbrace{b],\dots ],b}]},b\bigr)\\
 &=\omega\bigl(\bigl({\mathscr D}+{\mathscr D}^*\bigr)(a),b^{[p]}\bigr)-\omega(({\mathscr D}+{\mathscr D}^*)([[\dots [a,\overset{p-1\text{ terms}}{\overbrace{b],\dots ],b}]},b)\\
 &=-\omega\bigl(\bigl({\mathscr D}+{\mathscr D}^*\bigr)\bigl(b^{[p]}\bigr),a\bigr)+\omega(({\mathscr D}+{\mathscr D}^*)(b),\mathrm{ad}_b^{p-1}(a)).\tag*{\qed}
\end{align*}
\renewcommand{\qed}{}
\end{proof}
\section{Restricted orthosymplectic double extensions}\label{results}
 \subsection[D\_0 -extensions]{$\boldsymbol{{\mathscr D}_{\bar 0}}$-extensions}\label{D0omega0}

Let $(\mathfrak{a}, \omega_\mathfrak{a})$ be a~restricted orthosymplectic quasi-Frobenius Lie superalgebra, and let ${\mathscr D}\in \mathfrak{der}_{\bar 0}^p
(\mathfrak{a})$ be a restricted derivation.

Consider the map $ P\colon\mathfrak{a}_{\bar 0} \mapsto \mathbb{K} $ defined as in \eqref{mapP} and
satisfying the conditions \eqref{CondP1} and \eqref{CondP2}. We have shown in Lemma \ref{*property} that $P$ has the $*$-property with respect to the cocycle $C$ given by~\eqref{CEcocycle}.

Consider now the following two
maps
\begin{gather*}
\Omega\colon\ \mathfrak{a}\wedge \mathfrak{a} \rightarrow \mathbb K, \qquad (a,b) \mapsto \omega_\mathfrak{a}( {\mathscr D} \circ {\mathscr D}(a) + 2 {\mathscr D}^* \circ {\mathscr D}(a)+ {\mathscr D}^* \circ {\mathscr D}^*(a) +\lambda ( {\mathscr D} + {\mathscr D} ^* )(a) , b) , \\
T\colon\ \mathfrak{a}_{\bar 0} \rightarrow \mathbb K, \qquad a \mapsto \omega_\mathfrak{a}(({\mathscr D}+{\mathscr D}^*)(a),\bigl(\mathrm{ad}_a^\mathfrak{a}\bigr)^{p-2}({\mathscr D}(a)))+\lambda P(a).
\end{gather*}

Let us suppose that $T$ satisfies the $*$-property with respect to $\Omega$; and that $(\Omega, T)$ is in $B^2_*(\mathfrak{a};\mathbb{K})$.
Let us write $(\Omega, T)= \bigl(d^1_{\mathrm{CE}} (\chi), \mathrm{ind}^1(\chi) \bigr)$ for some $\chi\in C^1_*(\mathfrak{a}; \mathbb{K})$; namely, $\Omega(a,b)=\chi\bigl([a,b]_\mathfrak{a}\bigr)$ for every $a,b\in \mathfrak{a}$ and $T(a)=\chi(a^{[p]})$ for every $a\in \mathfrak{a}_{\bar 0}$.
 Since $\omega_\mathfrak{a}$ is non-degenerate, there exists~$Z_\Omega\in \mathfrak{a}$ such that
\begin{equation}
\label{Eazero}
\Omega(a,b)=\omega_\mathfrak{a}(Z_\Omega,[a,b]_\mathfrak{a}),\qquad \forall a,b\in \mathfrak{a} \qquad \text{and} \qquad T(a)=\omega_\mathfrak{a}\bigl(Z_\Omega,a^{[p]}\bigr), \qquad \forall a\in \mathfrak{a}_{\bar 0}.
\end{equation}

\begin{Theorem}[${\mathscr D}_{\bar 0}$-extension -- the case where $\omega$ is orthosymplectic]\label{MainTh} Let $\mathfrak{a}$ be a~restricted orthosymplectic quasi-Frobenius Lie superalgebra. Let ${\mathscr D}\in \mathfrak{der}_{\bar 0}^p(\mathfrak{a})$ be a restricted derivation satisfying the $p$-property \eqref{SConp}. Suppose further that
\[
(\Omega,T) \in B^2_*(\mathfrak{a}; \mathbb{K}) \qquad \text{and} \qquad (C, P)\in Z^2_*(\mathfrak{a}; \mathbb{K}).
\]
\begin{itemize}\itemsep=0pt
\item[$(i)$] There exists a~Lie superalgebra structure on $\mathfrak{g}:=\mathscr{K} \oplus \mathfrak{a} \oplus \mathscr{K} ^*$, where $\mathscr{K} :=\mathrm{Span}\{x\}$ for~$x$ even, defined as follows $($for any $a, b\in \mathfrak{a})$:
\begin{gather*}
[x ,x^*]_\mathfrak{g}=\lambda x, \qquad [a,b]_\mathfrak{g} := [a,b]_\mathfrak{a} + \omega_\mathfrak{a}( {\mathscr D}(a)+{\mathscr D}^*(a), b)x, \\
 [x^*,a]_\mathfrak{g} := {\mathscr D}(a)+ \omega_\mathfrak{a}(Z_\Omega,a)x,
\end{gather*}
where $\lambda\in \mathbb{K}$ and $Z_\Omega\in \mathfrak{a}$ is as in equation~\eqref{Eazero}. There exists a~closed anti-symmetric orthosymplectic form $\omega_\mathfrak{g}$ on $\mathfrak{g}$ defined as follows:
\begin{gather*}
{\omega_\mathfrak{g}}\vert_{\mathfrak{a} \times \mathfrak{a}}:= \omega_\mathfrak{a}, \qquad \omega_\mathfrak{g}(\mathfrak{a},\mathscr{K} ):=\omega_\mathfrak{g}(\mathfrak{a},\mathscr{K} ^*):=0, \\ \omega_\mathfrak{g}(x^*,x):=1, \qquad
\omega_\mathfrak{g}(x,x):=\omega_\mathfrak{g}(x^*,x^*) :=0.
\end{gather*}
\item[$(ii)$] There exists a $p|2p$-map on the double extension $\mathfrak{g}$ of $\mathfrak{a}$ given by
\begin{gather*}
a^{[p]_\mathfrak{g}} = a^{[p]_\mathfrak{a}} +P(a) x,\qquad
(x^*)^{[p]_\mathfrak{g}} = \gamma x^*+ a_0 + \tilde \lambda x,\qquad
x^{[p]_\mathfrak{g}} = b_0+ \sigma x +\delta x^*,
\end{gather*}
where
\begin{itemize}\itemsep=0pt
\item[$(a)$] The case where $\lambda\not =0$:
\begin{gather*}
 {\mathscr D}(a_0) = 0, \qquad \tilde \lambda = \frac{1}{\lambda} \omega(Z_\Omega, a_0), \qquad\gamma = \lambda^{p-1}, \qquad \delta=0, \\
 {\mathscr D}(b_0) = 0, \qquad \sigma = \frac{1}{\lambda} \omega(Z_\Omega, b_0), \qquad{\mathscr D}^*(b_0) = 0, \qquad b_0\in {\mathfrak z}(\mathfrak{a}),
\end{gather*}
and
\begin{equation}
\label{eqst1}
{\mathscr D}^*(a_0)=\sum_{1\leq i \leq p-1} (-1)^{p-1-i} \lambda^{p-1-i} {{\mathscr D}^*}^i(Z_\Omega).
\end{equation}

\item[$(b)$] The case where $\lambda=0$ and ${\mathscr D}\not =-\delta^{-1}\mathrm{ad}_{b_0}$:
\begin{gather*}
{\mathscr D}(a_0) = 0, \qquad \omega(Z_\Omega, a_0)=0, \qquad\delta = 0, \\
 {\mathscr D}(b_0) = 0,\qquad \omega(Z_\Omega, b_0)=0, \qquad {\mathscr D}^*(b_0) = 0, \qquad b_0\in {\mathfrak z}(\mathfrak{a}),
 \end{gather*}
and
\begin{equation}
\label{eqst2}
{\mathscr D}^*(a_0)+\gamma Z_\Omega= {{\mathscr D}^*}^{p-1}(Z_\Omega).
\end{equation}

\item[$(c)$] The case where $\lambda=0$ and ${\mathscr D}=-\delta^{-1}\mathrm{ad}_{b_0}$ is inner:
\begin{gather*}
{\mathscr D}(a_0) = 0, \qquad \omega(Z_\Omega, a_0)=0, \\
 {\mathscr D}(b_0) = 0, \qquad\omega(Z_\Omega, b_0)=0, \qquad {\mathscr D}^*(b_0) = -\delta Z_\Omega,
 \end{gather*}
and
\begin{equation}
\label{eqst3}
{\mathscr D}^*(a_0)+\gamma Z_\Omega= {{\mathscr D}^*}^{p-1}(Z_\Omega).
\end{equation}
\end{itemize}
\end{itemize}
\end{Theorem}
The orthosymplectic quasi-Frobenius Lie superalgebra $(\mathfrak{g}, \omega_\mathfrak{g})$ will be called the {\it symplectic double extension} by means of a 1-dimensional space of the orthosymplectic quasi-Frobenius Lie superalgebra $(\mathfrak{a}, \omega_\mathfrak{a})$. In the particular case where $x$ is central in $\mathfrak{g}$, namely $\lambda=0$, then the double extension was called {\it classical} in \cite{BaBe, DM1}.

\begin{proof} The proof of part (i) is in \cite{BM}, so we omit it. Let us prove part (ii), using Jacobson's theorem. We have
\begin{align*}
\mathrm{ad}_{(x^*)^{[p]_\mathfrak{g}}}^\mathfrak{g}(x)-\bigl(\mathrm{ad}_{x^*}^\mathfrak{g}\bigr)^p(x) &= \big[a_0+\gamma x^*+\tilde \lambda x,x\big]_\mathfrak{g}-\bigl(\mathrm{ad}_{x^*}^\mathfrak{g}\bigr)^{p-1}(-\lambda x)\\
 &= -\lambda \gamma x-(-1)^p \lambda^px=0.
\end{align*}
Additionally,
\begin{align*}
\bigl(\mathrm{ad}_{(x^*)^{[p]_\mathfrak{g}}}^\mathfrak{g}\bigr)(x^*)-\bigl(\mathrm{ad}_{x^*}^\mathfrak{g}\bigr)^p(x^*) &= \big[a_0+\gamma x^*+\tilde \lambda x,x^*\big]_\mathfrak{g}-(\mathrm{ad}_{x^*}^\mathfrak{g})^{p-1}\bigl([x^*,x^*]_\mathfrak{g}\bigr)\\
 &= -(\mathscr{D}(a_0)+\omega_\mathfrak{a}(Z_\Omega,a_0)xx)+\tilde \lambda \lambda x\\
 &= -\mathscr{D}(a_0)+\bigl(\tilde \lambda \lambda-\omega_\mathfrak{a}(Z_\Omega,a_0)\bigr)x\\
 &=0.
\end{align*}
Moreover,
\begin{align*}
\mathrm{ad}_{(x^*)^{[p]_\mathfrak{g}}}^\mathfrak{g}(a)-\bigl(\mathrm{ad}_{x^*}^\mathfrak{g}\bigr)^p(a) ={}& \big[a_0+\gamma x^*+\tilde \lambda x,a\big]_\mathfrak{g}-\bigl(\mathrm{ad}_{x^*}^\mathfrak{g}\bigr)^{p-1}\bigl(\big[x^*,a\big]_\mathfrak{g}\bigr)\\
={}& [a_0,a]_\mathfrak{g}+\gamma\big[ x^*,a\big]_\mathfrak{a}+-\bigl(\mathrm{ad}_{x^*}^\mathfrak{g}\bigr)^{p-1}\bigl(\mathscr{D}(a)+\omega_\mathfrak{a}(Z_\Omega,a)\bigr)\\
={}& [a_0,a]_\mathfrak{a}+\omega_\mathfrak{a}\bigl(\bigl({\mathscr D}+{\mathscr D}^*\bigr)(a_0),a\bigr)x+\gamma\bigl({\mathscr D} (a)+ \omega_\mathfrak{a}(a_0,a)x\bigr)\\
&- \bigg(\mathscr{D}^p(a)- \omega_\mathfrak{a}\bigg( Z_\Omega, \sum_{i=0}^{p-1} (-1)^{p-1-i}\lambda^{p-1-i} {\mathscr D}^i(a)\bigg)\bigg)\\
={}&0,
\end{align*}
as ${\mathscr D}$ satisfies the $p$-property and \eqref{eqst1}, or \eqref{eqst2} or \eqref{eqst3}.

On the other hand,
\[
\mathrm{ad}_{a^{[p]_\mathfrak{g}}}^\mathfrak{g}(x)-\big(\mathrm{ad}_{a}^\mathfrak{g}\big)^p(x) = 0.
\]
Additionally,
\begin{align*}
\mathrm{ad}_{a^{[p]_\mathfrak{g}}}^\mathfrak{g}(x^*)-\bigl(\mathrm{ad}_{a}^\mathfrak{g}\bigr)^p(x^*) ={}& \big[a^{[p]_\mathfrak{a}}+P(a)x, x^*\big]_\mathfrak{g}-(\mathrm{ad}_{a}^\mathfrak{g})^{p-1}\bigl(-\mathscr{D}(a)-\omega_\mathfrak{a}(Z_\Omega,a)x\bigr)\\
={}& -{\mathscr D}\bigl(a^{[p]_\mathfrak{a}}\bigr)-\omega_\mathfrak{a}\bigl(Z_\Omega, a^{[p]_\mathfrak{a}}\bigr)x+\lambda P(a)x\\
&{}-\bigl(\mathrm{ad}_a^\mathfrak{a}\bigr)^p\circ {\mathscr D}(a)-\omega_\mathfrak{a}\bigl(\bigl({\mathscr D}+{\mathscr D}^*\bigr)(a), \bigl(\mathrm{ad}_a^\mathfrak{a}\bigr)^{p-2}\circ {\mathscr D}(a)\bigr)x\\
={}&0,
\end{align*}
as ${\mathscr D}$ is a restricted derivation, and $(\Omega, T)$ is a coboundary.

Moreover,
\begin{align*}
\mathrm{ad}_{a^{[p]_\mathfrak{g}}}^\mathfrak{g}(b)-(\mathrm{ad}_{a}^\mathfrak{g})^p(b) ={}& \big[a^{[p]_\mathfrak{a}}+P(a)x,b\big]_\mathfrak{g}-\bigl(\mathrm{ad}_a^\mathfrak{g}\bigr)^{p-1}([a,b]_\mathfrak{g})\\
={}& \big[a^{[p]_\mathfrak{a}},b\big]_\mathfrak{a}+\omega_\mathfrak{a}\bigl(({\mathscr D}+{\mathscr D}^*)\bigl(a^{[p]_\mathfrak{a}}\bigr), b\bigr)x- \bigl(\mathrm{ad}_a^\mathfrak{a}\bigr)^{p-1}(b)\\
&-\omega_\mathfrak{a}\bigl(({\mathscr D}+{\mathscr D}^*)(a), \bigl(\mathrm{ad}_a^\mathfrak{a}\bigr)^{p-1}(b)\bigr)x\\
={}&0,
\end{align*}
as the Lie superalgebra $\mathfrak{a}$ is restricted, $(C,P)$ is a 2-cocycle and hence Proposition~\ref{CP2cocycle} can be applied.

Finally, using the same techniques as before we can show that $\mathrm{ad}_{x^{[p]_\mathfrak{g}}}^\mathfrak{g}=\big(\mathrm{ad}_x^\mathfrak{g}\big)^p$.
\end{proof}
\begin{Theorem}[converse of Theorem \ref{MainTh}]
\label{Rec1}
Let $(\mathfrak{g},\omega_\mathfrak{g})$ be a~restricted orthosymplectic quasi-Frobenius Lie superalgebra. Suppose there exists an even non-zero $x\in ([\mathfrak{g}, \mathfrak{g}]_\mathfrak{g})^\perp$ such that $\mathscr{K}:=\mathrm{Span}\{x\}$ is an ideal, and $\mathscr{K}^\perp$ is a $p$-ideal. Then, $(\mathfrak{g},\omega_\mathfrak{g})$ is obtained as a symplectic ${\mathscr D}_{\bar 0}$-extension by a 1-dimensional space from a~restricted orthosymplectic quasi-Frobenius Lie superalgebra $(\mathfrak{a},\omega_\mathfrak{a})$. Moreover, if $\mathfrak{z}_{\bar 0}(\mathfrak{g})\not =0$, then we can choose $x \in \mathfrak{z}_{\bar 0}(\mathfrak{g})$, so the double extension is classical.

\end{Theorem}
\begin{proof}

It has been proved in \cite{BM} that the space $\mathscr{K}^\perp$ is an ideal in $(\mathfrak{g},\omega_\mathfrak{g})$, and that there exists $x^* \in \mathfrak{g}_{\bar 0}$ \big(since $\bigl(\mathscr{K}^\perp\bigr)_{\bar 1}=\mathfrak{g}_{\bar 1}$\big) such that
\[
\mathfrak{g}=\mathscr{K}^\perp\oplus \mathscr{K}^*, \qquad \text{where} \ \mathscr{K}^*:=\mathrm{Span}\{x^*\}.
\]
We can normalize $x^*$ so that $\omega_\mathfrak{g}(x^*,x)=1$.

Let us define $\mathfrak{a}:=(\mathscr{K} +\mathscr{K}^*)^\perp$. We then have a decomposition $\mathfrak{g}=\mathscr{K} \oplus \mathfrak{a} \oplus \mathscr{K}^*$.

Let us define an~orthosymplectic form on $\mathfrak{a}$ by setting
\[
\omega_\mathfrak{a}={\omega_\mathfrak{g}}\vert_{\mathfrak{a}\times \mathfrak{a}}.
\]
It has been proved in \cite{BM} that the vector space $\mathfrak{a}$ can be endowed with a Lie superalgebra structure, and there exists an~orthosymplectic structure on $\mathfrak{a}$ for which $\mathfrak{g}$ is its symplectic double extension by means of the form $\omega_\mathfrak{a}$, a derivation $\mathscr D$ and $Z_\Omega\in \mathfrak{a}$ as in part (i) of Theorem \ref{MainTh}. In particular, it has been shown that the map
\[
\Omega\colon\ \mathfrak{a}\wedge \mathfrak{a} \rightarrow \mathbb K, \qquad (a,b) \mapsto \omega_\mathfrak{a}( {\mathscr D} \circ {\mathscr D}(a) + 2 {\mathscr D}^* \circ {\mathscr D}(a)+ {\mathscr D}^* \circ {\mathscr D}^*(a) +\lambda ( {\mathscr D} + {\mathscr D} ^* )(a) , b)
\]
is in $B^2_{\text{CE}}(\mathfrak{a}; \mathbb K)$, which implies that
\begin{equation}\label{Omegatr}
\Omega(a,b)=\omega_\mathfrak{a}(Z_\Omega, [a,b]_\mathfrak{a}) \qquad \text{for some} \ Z_\Omega\in \mathfrak{a}.
\end{equation}
Moreover, the map
\[
C\colon\ \mathfrak{a}\wedge \mathfrak{a} \rightarrow \mathbb K, \qquad (a,b) \mapsto \omega_\mathfrak{a}(( {\mathscr D} + {\mathscr D}^*)(a) , b)
\]
is in $Z_{\mathrm{CE}}^2(\mathfrak{a}; \mathbb K)$ by Lemma \ref{Ccocycle}.

It remains to show that there is a $[p|2p]$-map on $\mathfrak{a}$. Since $\mathfrak{a} \subset \mathscr{K}^\perp$ and ${\mathscr K}^\perp$ is a $p$-ideal, then
\[
a^{[p]_\mathfrak{g}}\in \mathscr{K}^\perp=\mathscr{K} \oplus \mathfrak{a} \qquad \text{ for any} \ a \in \mathfrak{a}_{\bar 0}.
\]
It follows that
\begin{equation*}
a^{[p]_\mathfrak{g}}=s(a)+P(a)x.
\end{equation*}
We will show that the map
\[
s\colon \ \mathfrak{a}_{\bar 0} \rightarrow \mathfrak{a}_{\bar 0}, \qquad a\mapsto s(a)
\]
is a $[p|2p]$-mapping on $\mathfrak{a}$.
Since $(\delta a)^{[p]_\mathfrak{g}}=\delta ^p\bigl(a^{[p]_\mathfrak{g}}\bigr)$, for all $\delta \in \mathbb K$ and for all $a\in \mathfrak{a}_{\bar 0} $, it follows that
\begin{gather}
s(\delta a) = \delta^p s(a),\label{rescond1rec}\\
P(\delta a) = \delta^p P(a).\label{Pcond1rec}
\end{gather}
Besides, for all $a\in \mathfrak{a}_{\bar 0}$ and for all $b\in \mathfrak{a}$, we have
\begin{align*}
\begin{split}
0&=\big[a^{[p]_\mathfrak{g}},b\big]_\mathfrak{g}-\bigl(\mathrm{ad}_a^\mathfrak{g}\bigr)^{p}(b)\\
&=[s(a),b]_\mathfrak{a}+\omega_\mathfrak{a}(({\mathscr D} + {\mathscr D}^*)(s(a)), b) x -\bigl(\mathrm{ad}_a^\mathfrak{a}\bigr)^p(b)- \omega_\mathfrak{a}\bigl(({\mathscr D} + {\mathscr D}^*)(a),\bigl(\mathrm{ad}_a^\mathfrak{a}\bigr)^{p-1}(b)\bigr) x.
\end{split}
\end{align*}
Therefore,
\begin{gather}
\omega_\mathfrak{a}(({\mathscr D}+{\mathscr D}^*)(s(a)), b) = \omega_\mathfrak{a}\bigl(({\mathscr D}+{\mathscr D}^*)(a),
\bigl(\mathrm{ad}_a^\mathfrak{a}\bigr)^{p-1}(b)\bigr) ,\label{cocycondrec}\\
[s(a),b]_\mathfrak{a} = \bigl(\mathrm{ad}_a^\mathfrak{a}\bigr)^p(b). \label{rescondrec}
\end{gather}
Now, since
\begin{align*}
 \sum_{1\leq i \leq p-1} is_i^\mathfrak{g}(a,b) \mu^{i-1}&=(\mathrm{ad}_{\mu a+b}^\mathfrak{g})^{p-1}(a)\\
&=(\mathrm{ad}_{\mu a+b}^\mathfrak{a})^{p-1}(a)+\omega_\mathfrak{a}\bigl(({\mathscr D}+{\mathscr D}^*)(\mu a+b),(\mathrm{ad}_{\mu a+b}^\mathfrak{a})^{p-2}(a) \bigr)x,
\end{align*}
it follows that
\[
s^\mathfrak{g}_i(a,b)=s_i^{\mathfrak{a}}(a,b)+\sigma_i^{\mathfrak{a}}(a,b)x.
\]
Moreover,
\begin{align*}
0={}&(a+b)^{[p]_\mathfrak{g}}-a^{[p]_\mathfrak{g}}-b^{[p]_\mathfrak{g}}- \sum_{1\leq i \leq p-1}s_i^\mathfrak{g}(a,b)\\
={}&\bigg( P(a+b)-P(a)-P(b)- \sum_{1\leq i \leq p-1}\sigma_i^{\mathfrak{a}}(a,b) \bigg) x\\
&{}+s(a+b)-s(a)-s(b)- \sum_{1\leq i \leq p-1}s_i^{\mathfrak{a}}(a,b).
\end{align*}
Consequently,
\begin{gather}
s(a+b)=s(a)+s(b)+ \sum_{1\leq i \leq p-1}s_i^{\mathfrak{a}}(a,b), \label{conres2rec}\\
P(a+b)-P(a)-P(b)= \sum_{1\leq i \leq p-1}\sigma_i^{\mathfrak{a}}(a,b).\nonumber 
 \end{gather}
Equations \eqref{rescond1rec}, \eqref{rescondrec}, \eqref{conres2rec} imply that $s$ defines a $[p|2p]$-map on $\mathfrak{a}$.

Now, since $\big[a^{[p]_\mathfrak{g}},x^*\big]_\mathfrak{g}=\bigl(\mathrm{ad}_a^\mathfrak{g}\bigr)^p(x^*)$, then
\begin{gather*}
P(a) \lambda x-\mathscr D (s(a)) - \omega_\mathfrak{a} (Z_\Omega, s(a))x \\
\qquad{} = - \bigl(\mathrm{ad}_a^\mathfrak{a}\bigr)^{p-1}\circ \mathscr D(a) - \omega_\mathfrak{a} (({\mathscr D}+{\mathscr D}^*)(a), \bigl(\mathrm{ad}_a^\mathfrak{a}\bigr)^{p-2} \circ {\mathscr D}(a))x.
\end{gather*}
It follows that
\begin{gather}
\mathscr D (s(a)) = \bigl(\mathrm{ad}_a^\mathfrak{a}\bigr)^{p-1}\circ \mathscr D(a),\label{resderrec} \\
P(a) \lambda = \omega_\mathfrak{a} (Z_\Omega, s(a)) - \omega_\mathfrak{a} (({\mathscr D}+{\mathscr D}^*)(a),\bigl(\mathrm{ad}_a^\mathfrak{a}\bigr)^{p-2} \circ {\mathscr D}(a)).\label{Pcbound}
\end{gather}
Equation~\eqref{resderrec} implies that ${\mathscr D}$ is a restricted derivation of $\mathfrak{a}$ (relative to the $[p|2p]$-map $s$). Let us defined the map
\[
T\colon\ a\rightarrow \mathbb K, \qquad a \mapsto \omega_\mathfrak{a} (({\mathscr D}+{\mathscr D}^*)(a),\bigl(\mathrm{ad}_a^\mathfrak{a}\bigr)^{p-2} \circ {\mathscr D}(a)) + \lambda P(a).
\]
We need the following lemma.
\begin{Lemma}
\[
(C,P)\in Z^2_{*} (\mathfrak{a}, \mathbb K)\qquad \text{ and }\qquad(\Omega,T)\in B^2_{*} (\mathfrak{a}, \mathbb K).
\]
\end{Lemma}
\begin{proof}
Equations \eqref{Pcond1rec} and \eqref{conres2rec} and Lemma \ref{*property}, imply that $P$ satisfies the $*$-property with respect to the cocycle $C$. Moreover, \eqref{cocycondrec} and Proposition~\ref{CP2cocycle} imply that $(C, P)\in Z^2_* (\mathfrak{a}; \mathbb K)$.

We have already showed in \cite{BM} that $\Omega \in B^2_{\text{CE}}(\mathfrak{a}; \mathbb K)$. Let us show that the
map $T$ satisfies the $*$-property with respect to $\Omega$. Indeed, for all $\delta \in \mathbb K$ and for all $a\in \mathfrak{a}_{\bar 0}$, we have
\begin{align*}
T(\delta a) & = \omega_\mathfrak{a} (({\mathscr D}+{\mathscr D}^*)(\delta a),(\mathrm{ad}_{\delta a}^\mathfrak{a})^{p-2} \circ {\mathscr D}(\delta a)) + \lambda P(\delta a) \\
 & = \delta^p \omega (({\mathscr D}+{\mathscr D}^*)(a),(\mathrm{ad}_{ a}^\mathfrak{a})^{p-2} \circ {\mathscr D}(a)) + \lambda \delta^p P(a) \\
 & =\delta^p T(a).
\end{align*}
Moreover, for all $a,b\in \mathfrak{a}_{\bar 0}$, we have
\begin{gather*}
T(a+b) \overset{\text{(by~\eqref{Pcbound})}}{=} \omega_\mathfrak{a}\bigl(Z_\Omega, (a+b)^{[p]_\mathfrak{a}}\bigr)\\
\qquad{}\overset{\text{(by definition) }}{=} \omega_\mathfrak{a}\bigl(Z_\Omega, a^{[p]_\mathfrak{a}}\bigr) + \omega_\mathfrak{a}\bigl(Z_\Omega, b^{[p]_\mathfrak{a}}\bigr)+ \omega_\mathfrak{a}\bigg(Z_\Omega, \sum_{1\leq i \leq p-1}s_i(a,b) \bigg) \\
\qquad{} \overset{\text{(by Lemma \ref{1lemma})}}{=}{} T(a)+T(b)
 + \sum_{\underset{x_{p-1}=b,~x_p=a}{x_k=a \text{ or } b}}\frac{1}{\pi(a)} \omega_\mathfrak{a}(Z_\Omega, [x_1,[x_2,[\dots ,[x_{p-1},x_p]_\mathfrak{a}\dots ]_\mathfrak{a}]_\mathfrak{a}]_\mathfrak{a}) \\
\qquad{}\overset{\text{(by~\eqref{Omegatr})}}{=} T(a)+T(b) + \sum_{\underset{x_{p-1}=b,~x_p=a}{x_k=a \text{ or } b}}\frac{1}{\pi(a)} \Omega ( x_1,[x_2,[\dots ,[x_{p-1},x_p]_\mathfrak{a}\dots ]_\mathfrak{a}]_\mathfrak{a}) \\
\qquad{} = T(a)+T(b) + \sum_{\underset{x_{p-1}=b,~x_p=a}{x_k=a \text{ or } b}}\frac{1}{\pi(a)} \Omega ( [[[x_p,x_{p-1}]_\mathfrak{a},x_{p-2}]_\mathfrak{a},\dots ,x_2]_\mathfrak{a}, x_1).
\end{gather*}
Now, equations~\eqref{Omegatr}, \eqref{Pcbound} imply that $(\Omega, T)\in B^2_*(\mathfrak{a}; \mathbb Z)$.
\end{proof}

Now, suppose that
\begin{gather*}
(x^*)^{[p]_\mathfrak{g}}=a_0+\beta x+\gamma x^*,\qquad \text{where}\ a_0\in \mathfrak{a} \ \text{and}\ \beta, \gamma \in \mathbb{K}.
\end{gather*}
We have
\[
0=\big[(x^*)^{[p]},x\big]_\mathfrak{g}-\bigl(\mathrm{ad}_{x^*}^\mathfrak{g}\bigr)^p(x)= -\gamma \lambda x+ \lambda^px.
\]
Therefore, $\gamma \lambda= \lambda^p$. Additionally,
\[
0=\big[(x^*)^{[p]},x^*\big]_\mathfrak{g}-\bigl(\mathrm{ad}_{x^*}^\mathfrak{g}\bigr)^p(x^*)= \beta \lambda x- ({\mathscr D}(a_0)+\omega_\mathfrak{a} (Z_\Omega, a_0)x).
\]
Therefore, ${\mathscr D}(a_0)=0$ and $\lambda \beta= \omega_\mathfrak{a} (Z_\Omega, a_0)$.

For all $a\in \mathfrak{a}$, we have
\begin{align*}
0={}&\big[(x^*)^{[p]_\mathfrak{g}},a\big]_\mathfrak{g}-\bigl(\mathrm{ad}_{x^*}^\mathfrak{g}\bigr)^p(a)\\
 ={}& [a_0,a]+ \omega_\mathfrak{a}(({\mathscr D}+{\mathscr D}^*)(a_0),a)x+\gamma {\mathscr D}(a) +\gamma \omega(Z_\Omega, a)x-{\mathscr D}^{p}(a) \\
 &- \sum_{1\leq i \leq p-1}\omega_\mathfrak{a}\bigl((-\lambda)^{p-1-i} {\mathscr D^*}^{i} (Z_\Omega),a\bigr)x - (-\lambda)^{p-1} \omega(Z_\Omega,a)x.
\end{align*}
It follows that ${\mathscr D}^p=\gamma {\mathscr D} +\mathrm{ad}_{a_0}$ and \smash{${\mathscr D}^*(a_0)+\gamma Z_\Omega= \sum_{0\leq i \leq p-1}(-\lambda)^{p-1-i} {\mathscr D^*}^{i} (Z_\Omega)$} ($\omega_\mathfrak{g}$ is non-degenerate).

Suppose now that
\[
x^{[p]_\mathfrak{g}}=b_0+\sigma x+\delta x^*,\qquad \text{where}\ \sigma,\delta\in \mathbb{K}\ \text{and}\ b_0\in \mathfrak{a}_{\bar 0}.
\]
We have
\[
0=\big[x^{[p]},x\big]_\mathfrak{g}-\bigl(\mathrm{ad}_{x}^\mathfrak{g}\bigr)^p(x)=- \delta \lambda x.
\]
Therefore, $\lambda \delta =0$. Additionally,
\[
0=\big[x^{[p]},x^*\big]_\mathfrak{g}-\bigl(\mathrm{ad}_{x}^\mathfrak{g}\bigr)^p(x^*)= \sigma \lambda x- ({\mathscr D}(b_0)+\omega_\mathfrak{a} (Z_\Omega, b_0)x).
\]
Therefore, ${\mathscr D}(b_0)=0$ and $\lambda \sigma-\omega(Z_\Omega,b_0)=0$.

Now, for any $b\in \mathfrak{a}$, we have
\[
0=\big[x^{[p]_\mathfrak{g}},b\big]_\mathfrak{g}-\bigl(\mathrm{ad}_x^\mathfrak{g}\bigr)^{p}(b)=[b_0,b]_\mathfrak{a}+\omega_\mathfrak{a}(({\mathscr D}+{\mathscr D}^*)(b_0), b)x+ \delta ( {\mathscr D}(b)+ \omega_\mathfrak{a}(Z_\Omega, b)x).
\]
It follows that $\delta {\mathscr D}(b)+[b_0,b]_\mathfrak{a}=0$ and ${\mathscr D}^*(b_0)+ \delta Z_\Omega=0$ (since $\omega_\mathfrak{a}$ is non-degenerate).

\textit{The case where $\lambda\not=0$.} It follows that $\delta=0$, ${\mathscr D}^*(b_0)=0$ and $b_0\in \mathfrak{z}(\mathfrak{a})$. Therefore, $\mathfrak{g}$ can be obtained from the restricted Lie superalgebra algebra $\mathfrak{a}$ as in Theorem \ref{MainTh}.

\textit{The case where $\lambda=0$.} It follows that $\omega(Z_\Omega, b_0)=0$. Here if $\delta=0$ then $b_0\in \mathfrak{z}(\mathfrak{a})$. Therefore, $\mathfrak{g}$ can be obtained from the restricted Lie superalgebra algebra $\mathfrak{a}$ as in Theorem \ref{MainTh}. However, if $\delta\not =0$, then it follows that ${\mathscr D}=-\delta^{-1} \mathrm{ad}_{b_0}$ and ${\mathscr D}^*(b_0)=-\delta Z_\Omega$.
\end{proof}

 \subsection[D\_1-extensions]{$\boldsymbol{{\mathscr D}_{\bar 1}}$-extensions}

Let $(\mathfrak{a}, \omega_\mathfrak{a})$ be a~restricted orthosymplectic quasi-Frobenius Lie superalgebra, and let ${\mathscr D}\in \mathfrak{der}_{\bar 1}(\mathfrak{a})$ be a derivation.

Consider the following two maps
\begin{gather*}
\Omega\colon\ \mathfrak{a}\wedge \mathfrak{a} \rightarrow \mathbb{K}, \qquad (a,b) \mapsto \omega_\mathfrak{a}( {\mathscr D} \circ {\mathscr D}(a) - {\mathscr D}^* \circ {\mathscr D}^*(a) , b), \\
\Delta\colon\ \mathfrak{a}_{\bar 0} \rightarrow \mathbb{K}, \qquad a \mapsto \omega_\mathfrak{a}(({\mathscr D}+{\mathscr D}^*)(a),\mathrm{ad}_a^{p-2}({\mathscr D}(a))).
\end{gather*}
Let us suppose that $(\Omega, \Delta)$ is in $B^2_*(\mathfrak{g};\mathbb{K})$ (i.e., a 2-coboundary in the restricted cohomology). Let us write $(\Omega, \Delta)= \bigl(d^1_{\mathrm{CE}} (\chi), \mathrm{ind}^1(\chi) \bigr)$ for some $\chi\in C^1_*(\mathfrak{a}; \mathbb{K})$; namely,
 \[
\Omega(a,b)=\chi([a,b]_\mathfrak{a}) \qquad \text{ and } \qquad \Delta(a)=\chi\bigl(a^{[p]}\bigr).
\]
Since $\omega_\mathfrak{a}$ is non-degenerate, there exists $a_0\in \mathfrak{a}$ such that
\begin{equation}
\label{EOazero}
\Omega(a,b)=\omega_\mathfrak{a}(a_0,[a,b]_\mathfrak{a}), \qquad \forall a,b \in \mathfrak{a} \qquad \text{and} \qquad \Delta(a)=\omega_\mathfrak{a}\bigl(a_0,a^{[p]}\bigr),\qquad \forall a\in \mathfrak{a}_{\bar 0}.
\end{equation}
\begin{Theorem}[${\mathscr D}_{\bar 1}$-extension -- the case where $\omega$ is orthosymplectic] \label{MainThO}
Let $(\mathfrak{a}, \omega_\mathfrak{a})$ be a~restricted orthosymplectic quasi-Frobenius Lie superalgebra. Let ${\mathscr D}\in \mathfrak{der}_{\bar 1}(\mathfrak{a})$ be a~restricted derivation such that the following conditions are satisfied
\[
(\Omega,\Delta) \in B^2_*(\mathfrak{a}; \mathbb{K}),
\]
together with
\[
{\mathscr D}^2=\mathrm{ad}_{a_0} \qquad \text{ and }\qquad {\mathscr D}(a_0)=0,
\]
where $a_0$ is as in \eqref{EOazero}.
\begin{itemize}
\item[$(i)$] There exists a Lie superalgebra structure on $\mathfrak{g}:=\mathscr{K} \oplus \mathfrak{a} \oplus \mathscr{K} ^*$, where $\mathscr{K} :=\mathrm{Span}\{x\}$ for~$x$ odd, defined as follows: $($for any $a, b\in \mathfrak{a})$
\begin{gather*}
[\mathscr{K} ,\mathfrak{g}]_\mathfrak{g}=0, \qquad [a,b]_\mathfrak{g} := [a,b]_\mathfrak{a} + \bigl(\omega_\mathfrak{a}( {\mathscr D}(a), b) +(-1)^{|a|} \omega_\mathfrak{a}(a, {\mathscr D}(b))\bigr)x, \\
[x^*, x^*]_\mathfrak{g}=2 a_0, \qquad [x^*,a]_\mathfrak{g} := {\mathscr D}(a) - \omega_\mathfrak{a}(a,a_0)x.
\end{gather*}
There exists a~closed anti-symmetric orthosymplectic form $\omega_\mathfrak{g}$ on $\mathfrak{g}$ defined as follows:
\begin{gather*}
{\omega_\mathfrak{g}}\vert_{\mathfrak{a} \times \mathfrak{a}}:= \omega_\mathfrak{a}, \qquad \omega_\mathfrak{g}(\mathfrak{a},\mathscr{K} ):=\omega_\mathfrak{g}(\mathfrak{a},\mathscr{K} ^*):=0, \\ \omega_\mathfrak{g}(x^*,x):=1, \qquad
\omega_\mathfrak{g}(x,x):=\omega_\mathfrak{g}(x^*,x^*) :=0.
\end{gather*}
\item[$(ii)$] There exists a $[p|2p]$-map on the double extension $\mathfrak{g}$ of $\mathfrak{a}$ given by
\[
a^{[p]_\mathfrak{g}} = a^{[p]_\mathfrak{a}}.
\]
\end{itemize}
\end{Theorem}
\begin{proof}
 Similar to that of Theorem \ref{MainTh}.
\end{proof}
\begin{Theorem}[converse of Theorem \ref{MainThO}]\label{Rec2}
Let $(\mathfrak{g},\omega_\mathfrak{g})$ be a~restricted orthosymplectic quasi-Frobenius Lie superalgebra. Suppose there exists a non-zero $x \in \mathfrak{z}(\mathfrak{g})_{\bar 1}$ such that $\omega(x,x)=~0$. Then, $(\mathfrak{g},\omega_\mathfrak{g})$ is a ${\mathscr D}_{\bar 1}$-extension of a~restricted orthosymplectic quasi-Frobenius Lie superalgebra~$(\mathfrak{a},\omega_\mathfrak{a})$.
\end{Theorem}

\begin{Remark}[the condition $\omega(x,x)=0$] The condition $\omega(x,x)=0$ is actually necessary. A~counterexample is the Lie superalgebra $C^1_{1/2}+A$, see \cite{BM}.
\end{Remark}

\begin{proof}[Proof of Theorem \ref{Rec2}]
Let $x$ be a non-zero element in $\mathfrak{z}(\mathfrak{g})_{\bar 1}$. It has been proved in \cite{BM} that the subspaces $\mathscr{K}:=\mathrm{Span}\{x\}$ and $\mathscr{K}^\perp$ are ideals in $(\mathfrak{g},\omega_\mathfrak{g})$. Since $\mathscr{K}$ is 1-dimensional and~$\omega(x,x)=0$, it follows that $\mathscr{K}\subset \mathscr{K}^\perp$ and $\dim\bigl(\mathscr{K}^\perp\bigr)=\dim(\mathfrak{g})-1$. Therefore, there exists~$x^* \in \mathfrak{g}_{\bar 1}$ such that
\[
\mathfrak{g}=\mathscr{K}^\perp\oplus \mathscr{K}^*, \qquad \text{where}\ \mathscr{K}^*:=\mathrm{Span}\{x^*\}.
\]
This $x^*$ can be normalized to have $\omega_\mathfrak{g}(x^*,x)=1$.

Let us define $\mathfrak{a}:=(\mathscr{K} +\mathscr{K}^*)^\perp$. We then have a decomposition $\mathfrak{g}=\mathscr{K} \oplus \mathfrak{a} \oplus \mathscr{K}^*$.

Let us define an~orthosymplectic form on $\mathfrak{a}$ by setting
\[
\omega_\mathfrak{a}={\omega_\mathfrak{g}}\vert_{\mathfrak{a}\times \mathfrak{a}}.
\]
It has been proved in \cite{BM} that the form $\omega_\mathfrak{a}$ is non-degenerate on $\mathfrak{a}$, and that $\mathfrak{g}$ is a symplectic double extension of $\mathfrak{a}$ by means of the form $\omega_\mathfrak{a}$ and a derivation ${\mathscr D}$ satisfying ${\mathscr D}^2=\mathrm{ad}_{a_0}$ and~${\mathscr D}(a_0)=0$. Moreover, it has been proved that the map $\Omega\in B^2_{\text{CE}}(\mathfrak{a}; \mathbb{K})$, and that the map~$C\in Z^2_{\text{CE}}(\mathfrak{a}; \mathbb{K})$.

It remains to show that there is a $[p|2p]$-map on $\mathfrak{a}$. Since $\mathfrak{a} \subset \mathscr{K}^\perp$ and since $x$ and $x^*$ are both odd, then
\[
a^{[p]_\mathfrak{g}}\in \mathfrak{a} \qquad \text{for any}\ a \in \mathfrak{a}_{\bar 0}.
\]
It follows that
\begin{equation*}
a^{[p]_\mathfrak{g}}=s(a).
\end{equation*}
We will show that the map
\[
s\colon\ \mathfrak{a}_{\bar 0} \rightarrow \mathfrak{a}_{\bar 0}, \qquad a\mapsto s(a),
\]
is a $[p|2p]$-map on $\mathfrak{a}$.
The fact that $(\delta a)^{[p]_\mathfrak{g}}=\delta^p\bigl(a^{[p]_\mathfrak{g}}\bigr)$ implies that $s(\delta a)=\delta^p s(a)$, for all $\delta\in \mathbb K$ and $a\in \mathfrak{a}_{\bar 0}$. Besides,
\begin{align*}
0&=\big[a^{[p]_\mathfrak{g}},b\big]_\mathfrak{g}-\bigl(\mathrm{ad}_a^\mathfrak{g}\bigr)^{p}(b)\\
&=[s(a),b]_\mathfrak{a}+\omega_\mathfrak{a}(({\mathscr D} + {\mathscr D}^*)(s(a)), b) x -\bigl(\mathrm{ad}_a^\mathfrak{a}\bigr)^p(b)- \omega_\mathfrak{a}\bigl(({\mathscr D} + {\mathscr D}^*)(a),\bigl(\mathrm{ad}_a^\mathfrak{a}\bigr)^{p-1}(b)\bigr) x.
\end{align*}
Therefore,
\begin{gather*}
\omega_\mathfrak{a}(({\mathscr D}+{\mathscr D}^*)(s(a)), b) = \omega_\mathfrak{a}\bigl(({\mathscr D}+{\mathscr D}^*)(a),\bigl(\mathrm{ad}_a^\mathfrak{a}\bigr)^{p-1}(b)\bigr),\\
[s(a),b]_\mathfrak{a} =\bigl(\mathrm{ad}_a^\mathfrak{a}\bigr)^p(b).
\end{gather*}
Now, since
\begin{align*}
 \sum_{1\leq i \leq p-1} is_i^\mathfrak{g}(a,b) \lambda^{i-1}&=\bigl(\mathrm{ad}_{\lambda a+b}^\mathfrak{g}\bigr)^{p-1}(a)\\
&=\bigl(\mathrm{ad}_{\lambda a+b}^\mathfrak{a}\bigr)^{p-1}(a)+\omega_\mathfrak{a}\bigl(({\mathscr D}+{\mathscr D}^*)(\lambda a+b),\bigl(\mathrm{ad}_{\lambda a+b}^\mathfrak{a}\bigr)^{p-2}(a) \bigr)x\\
&= \bigl(\mathrm{ad}_{\lambda a+b}^\mathfrak{a}\bigr)^{p-1}(a)
\end{align*}
it follows that
\[
s^\mathfrak{g}_i(a,b)=s_i^{\mathfrak{a}}(a,b).
\]
Moreover,
\begin{align*}
0&=(a+b)^{[p]_\mathfrak{g}}-a^{[p]_\mathfrak{g}}-b^{[p]_\mathfrak{g}}- \sum_{1\leq i \leq p-1}s_i^\mathfrak{g}(a,b)\\
&=s(a+b)-s(a)-s(b)- \sum_{1\leq i \leq p-1}s_i^{\mathfrak{a}}(a,b).
\end{align*}
Consequently,
\[
 s(a+b)=s(a)+s(b)+ \sum_{1\leq i \leq p-1}s_i^{\mathfrak{a}}(a,b).
\]
It follows that $s$ defines a $[p|2p]$-map on $\mathfrak{a}$. Now, since $\big[a^{[p]_\mathfrak{g}},x^*\big]_\mathfrak{g}=\bigl(\mathrm{ad}_a^\mathfrak{g}\bigr)^p(x^*)$, then
\[
-\mathscr D (s(a)) + \omega (s(a), a_0)x = - \bigl(\mathrm{ad}_a^\mathfrak{a}\bigr)^{p-1}\circ \mathscr D(a) - \omega \bigl(({\mathscr D}+{\mathscr D}^*)(a), \bigl(\mathrm{ad}_a^\mathfrak{a}\bigr)^{p-2} \circ {\mathscr D}(a)\bigr)x.
\]
It follows that
\[
\mathscr D (s(a)) = \bigl(\mathrm{ad}_a^\mathfrak{a}\bigr)^{p-1}\circ \mathscr D(a),\qquad \text{ and }\qquad
\omega(a_0, s(a)) = \omega \bigl(({\mathscr D}+{\mathscr D}^*)(a),\bigl(\mathrm{ad}_a^\mathfrak{a}\bigr)^{p-2} \circ {\mathscr D}(a)\bigr).
\]
The first equation shows that ${\mathscr D}$ is a restricted derivation of $\mathfrak{a}$ (relative to the $[p|2p]$-map~$s$). The second equation completes the proof that the map $(\Omega,\Delta)\in B^2_*(\mathfrak{a}; \mathbb{K})$, where $\Omega$ and $\Delta$ are given by
\begin{gather*}\label{EOOmega}
\Omega\colon\ \mathfrak{a}\wedge \mathfrak{a} \rightarrow \mathbb{K},\qquad (a,b) \mapsto \omega_\mathfrak{a}( {\mathscr D} \circ {\mathscr D}(a) - {\mathscr D}^* \circ {\mathscr D}^*(a) , b), \\
\Delta\colon\ \mathfrak{a}_{\bar 0} \rightarrow \mathbb{K},\qquad a \mapsto \omega_\mathfrak{a}\bigl(({\mathscr D}+{\mathscr D}^*)(a),\mathrm{ad}_a^{p-2}({\mathscr D}(a))\bigr).
\end{gather*}
Therefore, $\mathfrak g$ is a symplectic double extension of $\mathfrak{a}$.
\end{proof}

\section{Restricted periplectic double extensions}\label{sec4}

\subsection[D\_0-extensions]{$\boldsymbol{{\mathscr D}_{\bar 0}}$-extensions}\label{D0omega1}

Let $(\mathfrak{a}, \omega_\mathfrak{a})$ be a~restricted periplectic quasi-Frobenius Lie superalgebra, and let ${\mathscr D}\in \mathfrak{der}_{\bar 0}(\mathfrak{a})$ be a derivation.

Consider the following map
\[
\Omega\colon\ \mathfrak{a}\wedge \mathfrak{a} \rightarrow \mathbb{K}, \qquad (a,b) \mapsto \omega_\mathfrak{a} \bigl( ({\mathscr D} \circ {\mathscr D}+ 2 {\mathscr D}^* \circ {\mathscr D}+ {\mathscr D}^* \circ {\mathscr D}^* +\lambda {\mathscr D} + \lambda {\mathscr D} ^* )(a) , b \bigr).
\]
Let us suppose that $(\Omega, 0)$ is in $B^2_*(\mathfrak{a};\mathbb{K})$ (i.e., a 2-coboundary in the restricted cohomology). Let us write $(\Omega, 0)= \bigl(d^1_{\mathrm{CE}} (\chi), \mathrm{ind}^1(\chi) \bigr)$ for some $\chi\in C^1_*(\mathfrak{a}; \mathbb{K})$; namely,
\[
\Omega(a,b)=\chi([a,b]_\mathfrak{a}) \qquad \text{ and } \qquad \chi\bigl(a^{[p]}\bigr)=0.
\]
Since $\omega_\mathfrak{a}$ is non-degenerate, there exists $Z_\Omega\in \mathfrak{a}$ such that
\begin{equation*}
\Omega(a,b)=\omega_\mathfrak{a}(Z_\Omega,[a,b]_\mathfrak{a})\qquad \text{and} \qquad \omega_\mathfrak{a}\bigl(Z_\Omega,a^{[p]}\bigr)=0.
\end{equation*}

\begin{Theorem}[${\mathscr D}_{\bar 0}$-extension -- the case where $\omega$ is periplectic]\label{MainThOE} Let $(\mathfrak{a}, \omega_\mathfrak{a})$ be a restricted periplectic quasi-Frobenius Lie superalgebra. Let ${\mathscr D}\in \mathfrak{der}_{\bar 0}(\mathfrak{a})$ be a restricted derivation satisfying the $p$-property \eqref{SConp}. Suppose that the following condition is satisfied
\[
(\Omega, 0) \in B^2_*(\mathfrak{a}; \mathbb{K}).
\]
\begin{itemize}\itemsep=0pt
\item[$(i)$] There exists a~Lie superalgebra structure on $\mathfrak{g}:=\mathscr{K} \oplus \mathfrak{a} \oplus \mathscr{K} ^*$, where $\mathscr{K} :=\mathrm{Span}\{x\}$ for~$x$ odd and $\mathscr{K}^* :=\mathrm{Span}\{e\}$ for $e$ even, defined as follows: $($for any $a, b\in \mathfrak{a})$
\begin{gather*}
[x ,e]_\mathfrak{g}=\lambda x, \qquad [a,b]_\mathfrak{g} := [a,b]_\mathfrak{a} + (\omega_\mathfrak{a}( {\mathscr D}(a), b) + \omega_\mathfrak{a}(a, {\mathscr D}(b)))x, \\
 [e,a]_\mathfrak{g} := {\mathscr D}(a)+ \omega_\mathfrak{a}(Z_\Omega,a)x.
\end{gather*}
There exists a~closed anti-symmetric periplectic form $\omega_\mathfrak{g}$ on $\mathfrak{g}$ defined as follows:
\begin{gather*}
{\omega_\mathfrak{g}}\vert_{\mathfrak{a} \times \mathfrak{a}}:= \omega_\mathfrak{a}, \qquad \omega_\mathfrak{g}(\mathfrak{a},\mathscr{K} ):=\omega_\mathfrak{g}(\mathfrak{a},\mathscr{K} ^*):=0, \\ \omega_\mathfrak{g}(e,x):=1, \qquad
\omega_\mathfrak{g}(x,x):=\omega_\mathfrak{g}(e,e) :=0.
\end{gather*}
\item[$(ii)$] There exists a $[p|2p]$-map on symplectic the double extension $\mathfrak{g}$ of $\mathfrak{a}$ given by
\[
a^{[p]_\mathfrak{g}} = a^{[p]_\mathfrak{a}}, \qquad
e^{[p]_\mathfrak{g}} = a_0+\gamma e,
\]
where $a_0\in \mathfrak{a}$ and $\gamma \in \mathbb{K}$ are as in \eqref{SConp}, and the following conditions must be satisfied:
\begin{gather*}
\lambda \gamma = \lambda^p,\qquad \omega(Z_\Omega, a_0) = 0,\\
 \sum_{i+j=p-1}(-\lambda)^j({\mathscr D}^*)^{i} (Z_\Omega) = {\mathscr D}^*(a_0) + \gamma Z_\Omega.
\end{gather*}
\end{itemize}
\end{Theorem}
\begin{proof}
Similar to that of Theorem \ref{MainTh}.
\end{proof}

\begin{Theorem}[converse of Theorem \ref{MainThOE}]
\label{Rec3}
Let $(\mathfrak{g},\omega_\mathfrak{g})$ be a~restricted periplectic quasi-Frobenius Lie superalgebra. Suppose that there exists $0\not = x \in ([\mathfrak{g}, \mathfrak{g}])^\perp_{\bar 1}$ such that $\mathscr{K}:=\mathrm{Span}\{x\}$ is an ideal and $K^\perp$ is a $p$-ideal. Then $(\mathfrak{g},\omega_\mathfrak{g})$ is obtained a symplectic ${\mathscr D}_{\bar 0}$-extension from a~restricted periplectic quasi-Frobenius Lie superalgebra $(\mathfrak{a},\omega_\mathfrak{a})$.
\end{Theorem}
\begin{proof}
It has been proved in \cite{BM} that $\mathscr{K}^\perp$ is also an ideal in $(\mathfrak{g},\omega_\mathfrak{g})$. Since $\mathscr{K}$ is 1-dimensional and $\omega_\mathfrak{g}(x,x)=0$, it follows that $\mathscr{K}\subset \mathscr{K}^\perp$ and $\dim\bigl(\mathscr{K}^\perp\bigr)=\dim(\mathfrak{g})-1$. Therefore, there exists $e \in \mathfrak{g}_{\bar 0}$ \big(since $\bigl(\mathscr{K}^\perp\bigr)_{\bar 1}=\mathfrak{g}_{\bar 1}$\big) such that
\[
\mathfrak{g}=\mathscr{K}^\perp\oplus \mathscr{K}^*, \qquad \text{ where} \ \mathscr{K}^*:=\mathrm{Span}\{x^*\}.
\]
We can normalize $e$ so that $\omega_\mathfrak{g}(e,x)=1$.

Let us define $\mathfrak{a}:=(\mathscr{K} +\mathscr{K}^*)^\perp$. We then have a decomposition $\mathfrak{g}=\mathscr{K} \oplus \mathfrak{a} \oplus \mathscr{K}^*$. It has been proved in \cite{BM} that $\mathfrak{g}$ is a symplectic double extension of $\mathfrak{a}$.

Now, it remains to show that there is a $[p|2p]$-map on $\mathfrak{a}$. Since $\mathfrak{a} \subset \mathscr{K}^\perp$ and $\mathscr{K}^\perp$ is a $p$-ideal, then
\[
a^{[p]_\mathfrak{g}}\in \mathscr{K}^\perp=\mathscr{K} \oplus \mathfrak{a} \qquad \text{ for any}\ a \in \mathfrak{a}_{\bar 0}.
\]
It follows that $ a^{[p]_\mathfrak{g}}=s(a).$ We will show that the map
\[
s\colon\ \mathfrak{a}_{\bar 0} \rightarrow \mathfrak{a}_{\bar 0}, \qquad a\mapsto s(a),
\]
is a $[p|2p]$-map on $\mathfrak{a}$. The proof follows as in Theorem \ref{Rec1}.
\end{proof}
\subsection[D\_1-extensions]{$\boldsymbol{{\mathscr D}_{\bar 1}}$-extensions}
Let $(\mathfrak{a}, \omega_\mathfrak{a})$ be a~restricted periplectic quasi-Frobenius Lie superalgebra, and let ${\mathscr D}\in \mathfrak{der}_{\bar 1}(\mathfrak{a})$ be a derivation.

Consider the following map
\[
\Omega\colon \ \mathfrak{a}\wedge \mathfrak{a} \rightarrow \mathbb{K},\qquad (a,b) \mapsto \omega_\mathfrak{a}( {\mathscr D} \circ {\mathscr D}(a) - {\mathscr D}^* \circ {\mathscr D}^*(a) , b).
\]
Let us suppose that $(\Omega, 0)$ is in $B^2_*(\mathfrak{g};\mathbb{K})$ (i.e., a 2-coboundary in the restricted cohomology). Let us write $(\Omega,0)= \bigl(d^1_{\mathrm{CE}} (\chi), \mathrm{ind}^1(\chi) \bigr)$ for some $\chi\in C^1_*(\mathfrak{a}; \mathbb{K})$; namely,
\[
\Omega(a,b)=\chi([a,b]_\mathfrak{a}),\qquad \forall a,b\in \mathfrak{a} \qquad \text{ and } \qquad \chi\bigl(a^{[p]}\bigr)=0,\qquad \forall a\in \mathfrak{a}_{\bar 0}.
\]
Since $\omega_\mathfrak{a}$ is non-degenerate, there exists $a_0\in \mathfrak{a}$ such that
\begin{equation}
\label{azero}
\Omega(a,b)=\omega_\mathfrak{a}(a_0,[a,b]_\mathfrak{a})\qquad \text{and} \qquad \omega_\mathfrak{a}\bigl(a_0,a^{[p]}\bigr)=0.
\end{equation}

Consider now a map $P\colon \mathfrak{a}_{\bar 0} \mapsto \mathbb{K}$
satisfying the conditions (\ref{CondP1}) and (\ref{CondP2}). We have shown in Lemma \ref{*property} that $P$ has the $*$-property with respect to the cocycle $C$ given by equation~(\ref{CEcocycle}).
\begin{Theorem}[${\mathscr D}_{\bar 1}$-extension -- the case where $\omega$ is periplectic]\label{MainThOO} Let $(\mathfrak{a}, \omega_\mathfrak{a})$ be a~restricted periplectic quasi-Frobenius Lie superalgebra, and let ${\mathscr D}\in \mathfrak{der}_{\bar 1}(\mathfrak{a})$ be a derivation. Suppose that the following conditions are satisfied:
\[
(\Omega, 0) \in B^2_*(\mathfrak{a}, \mathbb{K}) \qquad \text{and} \qquad (C, P) \in Z^2_*(\mathfrak{a}, \mathbb{K}).
\]
Suppose further that $(a_0$ is given as in \eqref{azero}$)$
\[
{\mathscr D}^2=\mathrm{ad}_{a_0}, \qquad {\mathscr D}(a_0)={\mathscr D}^*(a_0)=0.
\]
\begin{itemize}
\item[$(i)$] There exists a Lie superalgebra structure on $\mathfrak{g}:=\mathscr{K} \oplus \mathfrak{a} \oplus \mathscr{K} ^*$, where $\mathscr{K} :=\mathrm{Span}\{x\}$ for~$x$ even and $\mathscr{K}^* :=\mathrm{Span}\{e\}$ for $e$ odd, defined as follows: $($for any $a, b\in \mathfrak{a})$
\begin{gather*}
[\mathscr{K} ,\mathfrak{g}]_\mathfrak{g}=0, \qquad [a,b]_\mathfrak{g} := [a,b]_\mathfrak{a} + \bigl(\omega_\mathfrak{a}( {\mathscr D}(a), b) +(-1)^{|a|}\omega_\mathfrak{a}(a, {\mathscr D}(b))\bigr)x, \\
[e, e]_\mathfrak{g}=2 a_0, \qquad [e,a]_\mathfrak{g} := {\mathscr D}(a) + \omega_\mathfrak{a}(a_0,a)x.
\end{gather*}
There exists a~closed anti-symmetric periplectic form $\omega_\mathfrak{g}$ on $\mathfrak{g}$ defined as follows:
\begin{gather*}
{\omega_\mathfrak{g}}\vert_{\mathfrak{a} \times \mathfrak{a}}:= \omega_\mathfrak{a}, \qquad \omega_\mathfrak{g}(\mathfrak{a},\mathscr{K} ):=\omega_\mathfrak{g}(\mathfrak{a},\mathscr{K} ^*):=0, \\ \omega_\mathfrak{g}(e,x):=1, \qquad
\omega_\mathfrak{g}(x,x):=\omega_\mathfrak{g}(e,e) :=0.
\end{gather*}
\item[$(ii)$] Suppose there exists $b_0\in \mathfrak{z}(\mathfrak{a})$ such that
\[
{\mathscr D}(b_0)={\mathscr D}^*(b_0)=\omega (a_0, b_0)=0.
\]

The $[p|2p]$-map on $\mathfrak{a}$ can be extended to $\mathfrak{g}$ as follows $(\mu$ is arbitrary$)$:
\[
\begin{array}{lcllcl}
a^{[p]_\mathfrak{g}} & = & a^{[p]_\mathfrak{a}} +P(a)x, \qquad
x^{[p]_\mathfrak{g}} & =& b_0+\mu x.
\end{array}
\]
\end{itemize}
\end{Theorem}
\begin{proof}
Similar to that of Theorem \ref{MainTh}.
\end{proof}

The converse of Theorem \ref{MainThOO} is given by the following theorem.

\begin{Theorem}[converse of Theorem \ref{MainThOO}]
\label{Rec4}
Let $(\mathfrak{g},\omega_\mathfrak{g})$ be a~restricted periplectic quasi-Frobenius Lie superalgebra. Suppose that $\mathfrak{z}(\mathfrak{g})_{\bar 0}\not=\{0\}.$ Then, $(\mathfrak{g},\omega_\mathfrak{g})$ is obtained as an ${\mathscr D}_{\bar 1}$-extension of a~restricted periplectic quasi-Frobenius Lie superalgebra $(\mathfrak{a},\omega_\mathfrak{a})$.

\end{Theorem}
\begin{proof}

Let $x$ be a non-zero element in $\mathfrak{z}(\mathfrak{g})_{\bar 0}$. It has been showed in \cite{BM} that both the subspaces~$\mathscr{K}:=\mathrm{Span}\{x\}$ and $\mathscr{K}^\perp$ are ideals in $(\mathfrak{g},\omega_\mathfrak{g})$. Since $\omega_\mathfrak{g}(x,x)=0$, it follows that $\mathscr{K}\subset \mathscr{K}^\perp$ and $\dim\bigl(\mathscr{K}^\perp\bigr)=\dim(\mathfrak{g})-1$. Therefore, there exists $e \in \mathfrak{g}_{\bar 1}$ \big(since $\bigl(\mathscr{K}^\perp\bigr)_{\bar 0}=\mathfrak{g}_{\bar 0}$\big) such that
\[
\mathfrak{g}=\mathscr{K}^\perp\oplus \mathscr{K}^*, \qquad \text{ where}\ \mathscr{K}^*:=\mathrm{Span}\{e\}.
\]
We can normalize $e$ so that $\omega_\mathfrak{g}(e,x)=1$.

Let us define $\mathfrak{a}:=(\mathscr{K} +\mathscr{K}^*)^\perp$. We then have a decomposition $\mathfrak{g}=\mathscr{K} \oplus \mathfrak{a} \oplus \mathscr{K}^*$.

We define a periplectic form on $\mathfrak{a}$ by setting $
\omega_\mathfrak{a}={\omega_\mathfrak{g}}\vert_{\mathfrak{a}\times \mathfrak{a}}.
$
The form $\omega_\mathfrak{a}$ is non-degenerate on~$\mathfrak{a}$.

It has been showed in \cite{BM} that the vector space $\mathfrak{a}$ can be endowed with a Lie superalgebra structure, and there exists a~periplectic structure on $\mathfrak{a}$ for which $\mathfrak{g}$ is its symplectic double extension. The proof now follows as in Theorem~\ref{MainThO}.\end{proof}
\section{Examples of symplectic double extensions}\label{sec5}

\subsection[The Lie superalgebra D\^{}7\_\{q,-q\} (for q not = 0,1), see cite{Ba}]{The Lie superalgebra $\boldsymbol{D^7_{q,-q}}$ (for $\boldsymbol{q\not = 0,1}$), see \cite{Ba}} Consider the Lie superalgebra $D^7_{q,-q}$, for $q\not = 0,1$, where the non-zero brackets are given in the basis $e_1, e_2\mid e_3, e_4$ as follows:
\[
[e_1, e_2]=e_2, \qquad [e_1, e_3]= q e_3, \qquad [e_1, e_4]=-q e_4.
\]
It has been shown in \cite{BM} that this Lie superalgebra is orthosymplectic quasi-Frobenius, where the form is given by\footnote{As in \cite{BM}, we adopt the following convention: $\langle e_i^*, e_j\rangle=\delta_{ij}$ and $\langle e_i^*\otimes e_j^*, e_k\otimes e_l\rangle= (-1)^{|e_k||e_j|} \langle e_i^*, e_k \rangle\langle e_j^*, e_l\rangle $.}
\[
\omega= e_1^*\wedge e_2^*+e_3^*\wedge e_4^*.
\]
There exists a $[p|2p]$-map given by $e_1^{[p]}=e_1$ and $e_2^{[p]}=0.$ Let us consider the even outer derivations
\[
{\mathscr D}_1= e_3\otimes e_3^*, \qquad {\mathscr D}_2=e_4\otimes e_4^*.
\]
Obviously, both of them are restricted derivations. An easy computation shows that ${\mathscr D}_1^*={\mathscr D}_2$. Besides, we have
\[
{\mathscr D}_1^2+ 2 {\mathscr D}_1\circ {\mathscr D}_1^*+{\mathscr D}_1^*\circ {\mathscr D}_1^* +\lambda({\mathscr D}_1+{\mathscr D}_1^*)=(\lambda+1)(e_3\otimes e_3^*+ e_4\otimes e_4^*).
\]
The cocycle $(\Omega,T)$ (see Theorem \ref{MainTh}) is a coboundary if and only if $\lambda=-1$. In this case, $\Omega \equiv 0$ and $C=e_3^*\wedge e_4^*$. Let us choose $Z_\Omega=u e_2$, where $u\in \mathbb K$. It follows that $P(e_1)=u$ and $P(e_2)=0$. Additionally,
\[
a_0=0, \qquad \tilde \lambda=0, \qquad \gamma=1, \qquad b_0=0, \qquad \sigma=0.
\]

\subsection[The Lie superalgebra C\_1\^{}1+A, see \protect{[2]}]{The Lie superalgebra $\boldsymbol{C_1^1+A}$, see \cite{Ba}}

 Consider the Lie superalgebra $C_1^1+A$ with the brackets given in the basis $e_1, e_2\mid e_3, e_4$ as follows:%
\[
[e_1, e_2]=e_2, \qquad [e_1, e_3]=e_3,\qquad [e_3, e_4]=e_2.
\]
It has been shown in \cite{BM} that this Lie superalgebra is orthosymplectic quasi-Frobenius with a~form given by
\[
\omega=e_1^*\wedge e_2^* - e_3^*\wedge e_4^*.
\]
The $[p|2p]$-map is given by
$
e_1^{[p]}=e_1$, $e_2^{[p]}=0$. There exists an odd outer derivation given by
\[
{\mathscr D}=e_1\otimes e_4^*+e_3\otimes e_2^*.
\]
A direct computation shows that ${\mathscr D}^*=-{\mathscr D}$. As a result, the cocycle $(C, \Delta)=(0,0)$ (see Theorem \ref{MainThO}). Additionally, since ${\mathfrak z}(\mathfrak{g})=0$, the condition ${\mathscr D}^2=\mathrm{ad}_{a_0}=0$ implies that $a_0=0$.
\subsection[The Lie superalgebra D\^{}5, see \protect{[2]}]{The Lie superalgebra $\boldsymbol{D^5}$, see \cite{Ba}}
Consider the Lie superalgebra $D^5$ with the brackets given in the basis $e_1, e_2, \mid e_3, e_4$ as follows
\[
[e_1, e_3]=e_3, \qquad [e_1, e_4]=e_4, \qquad [e_2, e_4]=e_3.
\]
It has been shown in \cite{BM} that Lie superalgebra is periplectic quasi-Frobenius with a form given by
\[
\omega=e_1^*\wedge e_3^*+e_2^*\wedge e_4^*.
\]
The $[p|2p]$-map is given by $
e_1^{[p]}=e_1$ and $e_2^{[p]}=0$. There exists an even outer derivative on $D^5$ given by
\[
{\mathscr D}=e_2\otimes e_2^*-e_4\otimes e_4^*.
\]
A direct computation shows that ${\mathscr D}^*=-{\mathscr D}$, and ${\mathscr D}^p={\mathscr D}$ so $\gamma=1$ and $a_0=0$ (see Theorem~\ref{MainThOE} for the definitions). As a result $\Omega=C=0$. Now, choose $\lambda$ any scalar that is a root to the equation $\lambda^p-\lambda=0$, and choose $Z_\Omega= 0$.
\subsection[The Lie superalgebra (2A\_\{1,1\}+2A)\^{}2\_\{1/2\}, see \protect{[2]}]{The Lie superalgebra $\boldsymbol{(2A_{1,1}+2A)^2_{1/2}}$, see \cite{Ba}}
Consider the Lie superalgebra $(2A_{1,1}+2A)^2_{1/2}$ with the following brackets given in the basis~$(e_1, e_2 \mid e_3, e_4)$:
\[
[e_3, e_3]=e_1, \qquad [e_4, e_4]=e_2, \qquad [e_3, e_4]=\tfrac{1}{2}(e_1+e_2).
\]
It has been shown in \cite{BM} that Lie superalgebra is periplectic quasi-Frobenius with a form given~by
\[
\omega=e_2^*\wedge e_3^*-e_1^*\wedge e_4^*.
\]
A $[p|2p]$-map can be defined as follows: $e_i^{[p]}=u_{i} e_1 + v_{i} e_2$, where $u_i$ and $v_i$ are scalar, for $i=1,2$.

There exists an outer odd derivation of this Lie superalgebra given by ${\mathscr D}=e_2\otimes e_3^*$. A direct computation shows that ${\mathscr D}^*={\mathscr D}$. It follows that $\Omega = \Delta= C=P=0$ (see Theorem \ref{MainThOO} for the definitions). In addition, $a_0=0$ and $b_0=\lambda_1 e_1+\lambda_2 e_2$ is arbitrary.

\subsection[The Witt algebra W(1) for p>3]{The Witt algebra $\boldsymbol{W(1)}$ for $\boldsymbol{p>3}$}
The Witt algebra, denoted by $W(1)$, is spanned by the generators $e_{-1},e_0,\dots ,e_{p-2}$ and endowed with the bracket
\begin{gather*}	
[e_i,e_j]=
 \begin{cases}
		(j-i)e_{i+j} &\text{if}\ i+j\in\{-1,\dots ,p-2\},\\
		0 &\text{otherwise.}
	\end{cases}
\end{gather*}
	
Let $\omega$ be an anti-symmetric bilinear form on $W(1)$, given by \smash{$\omega= \sum_{i,j=-1}^{p-2}\lambda_{i,j}e_i^*\wedge e_j^*$}, $\lambda_{i,j}\in\mathbb{K}$. Let us suppose further that $\omega$ is a $2$-cocycle, namely
\begin{equation}\label{wittcocycle}
 \omega(e_k,[e_i,e_j])+\omega(e_j,[e_k,e_i])+\omega(e_i,[e_j,e_k])=0.
\end{equation}
We will show that any anti-symmetric bilinear form on $W(1)$ satisfying (\ref{wittcocycle}) is degenerate. Suppose that $i$, $j$, $k$ are such that $i+j, i+k, j+k \in\{-1,\dots ,p-2\}.$ It follows
\begin{align*}
(\ref{wittcocycle})&\Longleftrightarrow \omega(e_k,(j-i)e_{i+j})+ \omega(e_j,(i-k)e_{i+k})+ \omega(e_i,(k-j)e_{j+k})=0\\
 &\Longleftrightarrow (j-i)\lambda_{k,i+j}+(i-k)\lambda_{j,i+k}+(k-j)\lambda_{i,j+k}=0.
\end{align*}

If we take $k=0$ in the above equation, we obtain
\[
 (j-i)\lambda_{0,i+j}+(i+j)\lambda_{i,j}=0.
\]

By taking $i=0$, the last equation reduces to
\[
 \lambda_{0,j}=0,\qquad \forall j\in\{-1,\dots ,p-2\}.
\]
That is to say, there is a row in the matrix of $\omega$ that contains only zeros. As a result, $\omega$ is degenerate.
\subsection[The Lie superalgebra K\^{}\{2,m\}, m odd, see \protect{[28]}]{The Lie superalgebra $\boldsymbol{K^{2,m}}$, $\boldsymbol{m}$ odd, see \cite{GKN}}
The Lie superalgebra $K^{2,m}$ is spanned by the generators $x_0, x_1\mid y_1,\ldots y_m$ (even $|$ odd), with non-zero brackets given by
\begin{gather*}
[x_0, y_i] = -[y_i, x_0] = y_{i+1},\qquad i\leq m-1, \\
[y_i, y_{m+1-i}] = [ y_{m+1-i}, y_i] = (-1)^{i+1}x_1, \qquad 1 \leq i\leq \frac{m+1}{2}.
\end{gather*}
It has been shown in \cite{GKN} that a Lie superalgebra of superdimension $n|m$ has a maximal nilindex $n + m-1$ only when $n=2$ and $m$ odd. Moreover, for any odd $m$, there is only one Lie superalgebra with this maximal nilindex and that is $K^{2,m}$.
\begin{Claim}\qquad
\begin{itemize}\itemsep=0pt
\item[$(i)$] For all $m\geq 3$, the Lie superalgebra $K^{2,m}$ is not periplectic quasi-Frobenius.
\item[$(ii)$] $K^{2,m}$ is orthosymplectic quasi-Frobenius if and only if $m=0 \bmod(p)$. Moreover, if ${m=0} \bmod(p)$, the form is given by
\[
x_0^* \wedge x_1^* - \frac{1}{2}\,y_1^* \wedge y_1^*- \frac{1}{2}\,(-1)^{\frac{m+3}{2}}y_{\frac{m+1}{2}}^* \wedge y_{\frac{m+3}{2}}^* - \sum_{1\leq i \leq \frac{m-3}{2}} i \, (-1)^{i+1}\; y_{i+1}^* \wedge y_{m+1-i}^*.
\]
\end{itemize}
\end{Claim}
\begin{proof} Let us prove part (i). Suppose that $K^{2,m}$ is periplectic quasi-Frobenius. The 2-cocycle condition applied to the periplectic form $\omega$ with respect to $\{x_0, x_1, y_i$\} gives
\begin{align*}
0&=\omega(x_0, [x_1, y_i])+ \omega (y_i, [x_0,x_1])+\omega (x_1, [y_i, x_0])\\
&=0+0-\omega(x_1, y_{i+1}), \qquad \text{ for}\qquad 1\leq i \leq m-1.
\end{align*}
Similarly, the 2-cocycle condition with respect to $\{y_1, y_2, y_{m-1} \}$ gives $\omega(x_1,y_1)=0$. Now, since~$\omega$ is odd it follows that $\omega(x_1, v)=0$ for all $v \in K^{2,m}$. Thus, $\omega$ is degenerate which is a~contradiction.

Let us proof part (ii). Since $\omega$ is even we will check the 2-cocycle condition only for the triple~$\{x_0, y_i, y_i\}$ and $\{x_0, y_i, y_j\}$ for $i<j$. Let us write
\[
\omega= \alpha \, x^*_0\wedge x^*_1 - \sum_{i,j} \alpha_{i,j} \, y^*_i \wedge y^*_j.\tag*{\qed}
\]\renewcommand{\qed}{}
\end{proof}

\textit{The case of $\{x_0, y_i, y_i\}$}. For $i \leq \frac{m+1}{2}$, the 2-cocycle condition is equivalent to
\begin{gather}
\alpha = 2(-1)^{\frac{m+3}{2}} \alpha_{\frac{m+1}{2}, \frac{m+3}{2}}\qquad \text{if} \ i=\frac{m+1}{2},\qquad
\alpha_{i,i+1} = 0\qquad \text{if}\ i < \frac{m+1}{2}.\label{K1}
\end{gather}
Moreover, if $i> \frac{m+1}{2}$ the 2-cocycle condition is equivalent to
\[
\alpha_{i,i+1}=0 \qquad \text{ for } \ \frac{m+1}{2} < i <m.
\]

\textit{The case of $\{x_0, y_i, y_j\}$, where $i<j$.} In the case of $1\leq i \leq \frac{m+1}{2}$ and $j=m+1-i$, the 2-cocycle condition implies that
\begin{gather}
\alpha = \alpha_{m,2} \qquad \text{if}\ i=1,\nonumber \\
(-1)^{i+1}\alpha - \alpha_{m+1-i,i+1}-\alpha_{i,m+2-i} = 0 \qquad \text{if}\ 1<i<\frac{m+1}{2}.\label{K2}
\end{gather}
On the other hand, if $j \neq m+1-i$, the 2-cocycle condition is equivalent to
\begin{gather*}
\alpha_{m,i+1} = 0\qquad \text{if}\ j=m, \ \text{and}\ i\not = 1,\\
\alpha_{j,i+1}+ \alpha_{i,j+1} = 0 \qquad \text{if} \ 1\leq i \leq \frac{m+1}{2}, \ \text{ and }\ j\not = m.
\end{gather*}
Now, in the case where $i> \frac{m+1}{2}$, the 2-cocycle condition is equivalent to
\begin{gather}
\alpha_{m,i} = 0 \qquad \text{if} \ j=m,\nonumber \\
\alpha_{j,i+1} + \alpha_{i,j+1} = 0 \qquad \text{if} \ \frac{m+1}{2} < i<j<m.\label{K5}
\end{gather}
Combining equations~(\ref{K1}) and (\ref{K2}) we get $m \alpha=0$. If $m \not =0 \bmod(p)$ then $\alpha=0$ and the form $\omega$ will be degenerate. If $m =0 \bmod(p)$, then the system of equations given by \eqref{K1}--\eqref{K5} is consistent and is not unique. Let us choose $\alpha_{\frac{m+1}{2}, \frac{m+3}{2}}=\frac{1}{2} (-1)^{\frac{m+3}{2}}$ which would imply that~$\alpha=1$. Now, equations~\eqref{K2} imply that $\alpha_{m+1-i, i+1}=i (-1)^{i+1}$ for $i=1,\ldots \frac{m-1}{2}$. Choosing~$\alpha_{1,1}=\frac{1}{2}$ and the remaining coefficients to be trivial produces a non-degenerate form.

\begin{Claim}
 $K^{2,m}$ is restricted if and only if $m\leq p$. Explicitly, the $[p|2p]$-map is given by
\[
x_0^{[p]}=s_1 x_1, \qquad x_1^{[p]}=s_2 x_1, \qquad\text{ where}\ s_1, s_2\in \mathbb K.
\]
\end{Claim}
\begin{proof} First, observe that the even part of $K^{2,m}$ is abelian. Let us write ${x_1^{[p]}=s_2 x_1+ \tilde s_2 x_0}$. Now, the condition $\big[x_1^{[p]},y_i\big]=\mathrm{ad}_{x_1}^p(y_i)$ implies that $\tilde s_2=0$. Therefore, $x_1^{[p]}=s_2 x_1$. Additionally, let us write $x_0^{[p]}=s_1 x_1+ \tilde s_1 x_0$. The condition $[x_0^{[p]},y_i]=\mathrm{ad}_{x_0}^p(y_i)$ implies that $\tilde s_1 y_{i+1}=y_{p+i}$ for $i<m$. This condition implies that $p+i>m$, for all $i<m$, and~$\tilde s_1=0$.

Hereafter, suppose that the $[p|2p]$-map is given by
\[
x_0^{[p]}=0, \qquad x_1^{[p]}= x_1.
\]
We will list a few outer derivations and study symplectic double extensions using those derivations.
\end{proof}
\subsubsection{A non-suitable derivation}
Consider the outer derivation given by
\[
{\mathscr D}= x_1\otimes x_1^*-\frac{p-1}{2} \sum_{1\leq i \leq p} y_i \otimes y_i^*.
\]
Now, ${\mathscr D}\bigl(x_1^{[p]}\bigr)={\mathscr D}( x_1)= x_1\not = \mathrm{ad}_{x_1}^{p-1} \circ {\mathscr D}(x_1)=0.$ Hence, this derivation is not suitable for a~double extension because it is not a restricted derivation.

\subsubsection{A derivation yielding a trivial extension} Consider the outer derivation given by
\[
{\mathscr D}=x_1\otimes x_0^*.
\]
This is obviously a restricted derivation. A direct computation shows that ${\mathscr D}^*=-{\mathscr D}$. It follows that the cocycle $C$ of Lemma \ref{Ccocycle} as well as the map $\Omega$ in Section \ref{D0omega0} are identically zero.

The conditions of Theorem \ref{D0omega0} are all satisfied. We can choose
\[
a_0=0,\qquad \lambda=0,\qquad \gamma=0,\qquad b_0=0,\qquad \sigma=0,\qquad Z_\Omega=0, \qquad P\equiv 0.
\]
\subsubsection{A derivation yielding a non-trivial extension} Consider the outer derivation given by
\[
{\mathscr D}=y_{p-1}\otimes y_1^*+y_p \otimes y_2^*.
\]
Obviously, this is a restricted derivation. A direct computation shows that
\[
{\mathscr D}^*= y_p \otimes y_2^* -2 y_1 \otimes y_3^*.
\]
It follows that
\[
{\mathscr D}^2+2 {\mathscr D}^*\circ {\mathscr D}+{{\mathscr D}^*}^2+\lambda({\mathscr D}+{\mathscr D}^*)=
\lambda y_{p-1}\otimes y_1^*+2\lambda y_p \otimes y^*_2-2\lambda y_1\otimes y^*_3.
\]
A direct computation shows that the map $\Omega$, defined in Section \ref{D0omega0}, has the following form (the zero terms are omitted):
\[
\Omega(y_1,y_3)=-2\lambda, \qquad \Omega(y_2,y_2)=2\lambda .
\]
\textit{The case where $p=3$.} In this case, the map $C$ is given by
\[
C=\frac{1}{2} \; y_1^*\wedge y_3^* +y_2^*\wedge y_2^*.
\]
The map $\Omega$ is a coboundary in the usual cohomology, namely $\Omega= \partial \phi$ with $\phi=\lambda \, x_1^*$. Let us choose (see Theorem \ref{MainTh}, for the definitions)
\[
Z_\Omega=\lambda x_0,\qquad \text{ and } \qquad P(a)=x^*_1\bigl(a^{[p]}\bigr) \text{ for $a \in K^{2,3}_{\bar 0}$}.
\]
On the other hand (see Theorem \ref{MainTh}, for the definitions),
\[
\gamma=\lambda^{p-1},\qquad a_0=- \gamma x_0, \qquad \tilde \lambda=0, \qquad b_0=x_1,\qquad \sigma=1.
\]
\textit{The case where $p>3$.} In this case, the map $C$ is given by
\[
C=2 \, y_1^*\wedge y_{3}^* - 2 \,y_2^*\wedge y_2^*.
\]
The map $\Omega$ cannot be a coboundary, except for $\lambda=0$ where it becomes identically trivial. In this case, we can choose
\[
\gamma=0, \qquad b_0=0, \qquad a_0= x_1,\qquad Z_\Omega= x_1.
\]

\subsection[The Lie superalgebra K\^{}\{2,m\}, m even, see \protect{[28]}]{The Lie superalgebra $\boldsymbol{K^{2,m}}$, $\boldsymbol{m}$ even, see \cite{GKN}}

For $m$ even, the Lie superalgebra $K^{2,m}$ is spanned by the generators $x_0, x_1\mid y_1,\ldots y_m$ (even $|$ odd), with non-zero brackets given by
\begin{gather*}
[x_0, y_i] = -[y_i, x_0] = y_{i+1},\qquad 1\leq i\leq m-1, \\
[y_i, y_{m-i}] = [ y_{m-i}, y_i] = (-1)^{(m-2i)/2}x_1, \qquad 1 \leq i\leq \frac{m}{2}.
\end{gather*}
It has been shown in \cite{GKN} that the superalgebra $K^{2,m}$ has maximal nilindex $m$.

The superalgebra $K^{2,m},$ where $m$ is even, is restricted if and only if $m\leq p$. In that case, the $[p|2p]$-map is given by
\[
x_0^{[p]}=s_1 x_1, \qquad x_1^{[p]}=s_2 x_1,\qquad \text{ where}\ s_1, s_2\in \mathbb K.
\] The proof is similar to the one for $m$ odd.

\begin{Claim}\quad
 \begin{itemize}\itemsep=0pt\item[$(i)$] For all $m\geq 2$, the Lie superalgebra $K^{2,m}$ is not periplectic quasi-Frobenius, except when ${p=3}$. In this case, the form is given by
\[\omega=x_{0}^* \wedge y_{2}^*+x_{1}^* \wedge y_{1}^*.\]
\item[$(ii)$]For all $m\geq 2$, the Lie superalgebra $K^{2,m}$ is orthosymplectic quasi-Frobenius if and only if $m<p$. If $m<p$, the form $\omega$ is given by
\[
\omega= \sum_{i=1}^{m/2} \alpha_i \; y_{i}^* \wedge y_{m-i+1}^* +2\, x_0^*\wedge x_1^*,
\]
where $\alpha_{i}= (-1)^{\frac{m}{2}-i+1} (m-2i+1)$.
\end{itemize}
\end{Claim}
\subsubsection{A derivation yielding a non-trivial extension} Consider the outer derivation given by
\[
{\mathscr D}=x_1\otimes x_0^*.
\]
This derivation is even and restricted. Moreover, it has the $p$-property (see~\eqref{SConp} for definitions) with $\gamma=0$ and $a_0=l_1x_1,~l_1\in\mathbb K$. Because $\gamma$ is zero, it follows that we only have to consider the case $(b)$ of Theorem \ref{MainTh}. Direct computations show that ${\mathscr D}^2=0$ and ${\mathscr D}^*=-{\mathscr D}$. It follows that $\sum_{i=1}^{p-1}\sigma_i^{\mathfrak{a}}(a,b)=0$ (see~\eqref{CondP2}), that $(C,P)$ is a restricted $2$-cocycle, and that the maps $\Omega$ and $T$ in Section \ref{D0omega0} are identically zero. Because $(\Omega,T)$ has to be a restricted coboundary, we can choose $Z_{\Omega}=tx_1+sy_m,~t,s\in\mathbb K,$ to obtain
\begin{gather*}
0=\omega(Z_{\Omega},[a,b])\qquad \text{ and }\qquad 0=\omega\bigl(Z_{\Omega},\tilde{a}^{[p]}\bigr),\\
\text{ for all }\qquad a,b\in K^{2,m}\qquad \text{ and }\qquad \tilde{a} \in K^{2,m}_{\bar{0}}.
\end{gather*}
 The conditions of Theorem \ref{MainTh} are satisfied and we can choose $b_0=l_2x_1,~l_2\in \mathbb K$ and $\sigma\in \mathbb K$ arbitrary to build a non-trivial double extension.

\section{Conclusions and outlook}
\begin{enumerate}\itemsep=0pt
 \item[(1)] In \cite{AM}, Novikov Lie algebras appearing within symplectic structures are discussed, and a~classification is offered in low dimensions. It is interesting to superize the results of \cite{AM}.

 \item[(2)] In \cite{D}, a certain class of free nilpotent Lie algebras possessing a symplectic structure is studied. The investigation is based on the fact that for a quasi-Frobenius Lie algebra we must have (see \cite{BC})
\[
\mathrm{dim}(\mathfrak{z}(\mathfrak{a}))\leq \mathrm{dim}( \mathfrak{a}/[\mathfrak{a}, \mathfrak{a}]).
\]
It is easy to supersize this formula, and the proof is similar to that given in \cite{BC}. It seems to be possible to construct a super-analog of the Lie algebras in \cite{D}, and their results can be superized. However, it is not clear yet whether these Lie (super)algebras are restricted or not.
On the other hand, the classification of restricted $p$-nilpotent Lie algebras in low dimension have been obtained in \cite{DU, MS}. It is interesting to superize this classification, and select among those Lie algebras that are quasi-Frobenius.

 \item[(3)] A nilpotent orthosymplectic quasi-Frobenius Lie superalgebra cannot be Frobenius. To see this, take the restriction of the form on the even part ${\mathfrak g}_{{\bar 0}}$ and use the result \cite[Proposition~6]{GR}. Engel's theorem is used to prove the result, which is also valid over a field with characteristic non-zero. It would be interesting to examine the case where the form is periplectic.

 \item[(4)] Simple modular Lie (super)algebras over a field of positive characteristic may have a~degenerate Killing form and some of them do not admit a Cartan matrix, for instance, those of Cartan type. Hence, the techniques used in \cite{GR} to show that real or complex semi-simple Lie algebras cannot be Frobenius will not work in characteristic $p$. Furthermore, we only know the classification of modular Lie algebras over a field of characteristic $p>3$, see \cite{S2}. The problem of classification remains open in the super setting. It is interesting to investigate which modular Lie superalgebras are Frobenius.

 \item[(5)] If the Lie superalgebra $\mathfrak{a}$ is Frobenius (i.e., $\omega_\mathfrak{a}=d_{\text{CE}}(f_\mathfrak{a}) $), can its double extension $\mathfrak{g}$ be Frobenius as well? In the case where $\omega $ is odd, then $\mathfrak{g}$ cannot be Frobenius. Indeed, suppose that $\omega_\mathfrak{g}=d_{\text{CE}}(f_\mathfrak{g})$. It follows that
\begin{gather*}
0=\omega_\mathfrak{g}(x^*, x^*)=f_\mathfrak{g} ([x^*, x^*])= f_\mathfrak{g} (2 a_0) \\ \text{while} \ 1= \omega_\mathfrak{g}(x^*, x)=f_\mathfrak{g} ([x^*, x])=f_\mathfrak{g} (2 a_0),
\end{gather*}
a contradiction. In the case where $\omega_\mathfrak{a}$ is even, the linear map $f_\mathfrak{g}$ must satisfy:
\begin{gather*}
f_\mathfrak{g}(x) = -\lambda^{-1}, \qquad f_\mathfrak{g} ({\mathscr D}(a))= \lambda^{-1} \omega (Z_\Omega, a), \\ f_\mathfrak{g}([a,b]_\mathfrak{a})=f_\mathfrak{a} ([a,b]_\mathfrak{a}) +\lambda^{-1} \omega_\mathfrak{a} ({\mathscr D} (a)+{\mathscr D}^* (a), b).
\end{gather*}
One could describe all Lie superalgebras $\mathfrak{a}$ for which $f_\mathfrak{g}$ exists.
\end{enumerate}

\subsection*{Acknowledgements}
We are grateful for the valuable feedback we received from the referees, which enabled us to significantly improve the presentation of this paper.

\pdfbookmark[1]{References}{ref}

\LastPageEnding

\end{document}